\documentclass[11pt]{amsart}
\usepackage{amsmath}
\newtheorem{thm}{Theorem}[section]
\newtheorem{cor}[thm]{Corollary}
\newtheorem{lem}[thm]{Lemma}
\newtheorem{prop}[thm]{Proposition}
\theoremstyle{definition}

\theoremstyle{remark}
\newtheorem{rem}[thm]{\bf Remark}
\newtheorem{exe}[thm]{\bf Example}

\numberwithin{equation}{section}

\begin{document}
\title[ACTION OF NON ABELIAN GROUP OF HOMOTHETIES ON $\mathbb{R}^{n}$]{ ACTION OF NON  ABELIAN  GROUP GENERATED BY AFFINE  HOMOTHETIES  ON
$\mathbb{R}^{n}$}

\author{Adlene Ayadi and Yahya N'dao}

\address{Adlene Ayadi, Department of Mathematics,
Faculty of Sciences of Gafsa, Gafsa, Tunisia.}
 \address{Yahya N'dao, Department of Mathematics, Faculty of Sciences of Gafsa,
Gafsa, Tunisia}

\email{ adleneso@yahoo.fr,\;\;\; yahiandao@yahoo.fr}

\thanks{This work is supported by the research unit: syst\`emes dynamiques et combinatoire:
99UR15-15} \subjclass[2000]{37C85} \keywords{Homothety, orbit,
density, minimal, non abelian, action, dynamic,...}

\begin{abstract}

In this paper, we study the action of non abelian group $G$
generated by affine homotheties  on\ $\mathbb{R}^{n}$. We prove
that $G$ satisfies one of the following properties: (i) there exist a  subgroup $\Lambda_{G}$
of $\mathbb{R}^{*}$ containing $0$ in its closure (i.e.  $0\in\overline{\Lambda_{G}}$), a $G$-invariant affine subspace $E_{G}$ of $\mathbb{R}^{n}$ and $a\in E_{G}$ such that
 $\overline{G(x)}=\overline{\Lambda_{G}}(x-a)+E_{G}$ for every $x\in \mathbb{R}^{n}$. In particular,
 $\overline{G(x)}=E_{G}$ for every $x\in E_{G}$ and every orbit in
$U=\mathbb{R}^{n}\backslash E_{G}$  is minimal in $U$. (ii) there exists  a closed subgroup $H_{G}$
of $\mathbb{R}^{n}$  and  $a\in \mathbb{R}^{n}$ such that for
every   $x\in\mathbb{R}^{n}$   we have
 $\overline{G(x)}=(x+H_{G})\cup(-x+a+ H_{G})$.
\end{abstract}
\maketitle

\section{\bf Introduction }

 A map  $f\ :  \mathbb{R}^{n}\longrightarrow
\mathbb{R}^{n}$  is called an affine homothety  if there exists
$\lambda\in\mathbb{R}\backslash\{-1,0,1\}$  and  $a\in
\mathbb{R}^{n}$  such that  $f(x)= \lambda (x-a)+a$  for every
 $x\in \mathbb{R}^{n}$. (i.e.  $f=T_{a}\circ (\lambda.
id_{\mathbb{R}^{n}})\circ T_{-a}$, \ $T_{a}:x\longmapsto x+a$).  Write $f= (a,\lambda)$  and we call \ $a$ \
{\it the center} of  $f$ and \ $\lambda$ \ the ratio of $f$. Denote by  $$\mathcal{H}(n,\mathbb{R}):=\left\{\ f:x\
\longmapsto\ \lambda x+a;\ a\in\mathbb{R}^{n},\
\lambda\in\mathbb{R}^{*} \right\}$$
 the affine  group generated by all affine homotheties of $\mathbb{R}^{n}$.
 We let  $\mathcal{T}_{n}(\mathbb{R})$   the group of
translation of $\mathbb{R}^{n}$. We say a {\it group of affine
homotheties} any subgroup of $\mathcal{H}(n,\mathbb{R})$.
\bigskip

 Denote by  $\mathcal{S}_{n}$   the subgroup of
$\mathcal{H}(n,\mathbb{R})$  of affine symmetries, i.e.
$$\mathcal{S}_{n}:=\left\{\ f:x\ \longmapsto\ \varepsilon x+a;\
a\in\mathbb{R}^{n},\ \varepsilon\in\{-1,1\}\right\}$$

Write $f=(a,\varepsilon)$, for every $f\in \mathcal{S}_{n}$
 defined by $f(x)=\varepsilon x +a$. Under above notation, a map
 $f=(a,\lambda)\in \mathcal{H}(n, \mathbb{R})$ is either an affine homothety if
 $\lambda\notin\{-1,0,1\}$, and here $a=f(a)$, or an affine
 symmetry (resp. translation) if
 $\lambda=-1$ (resp.  $\lambda=1$), and in this case $a=f(0)$.
\bigskip

 Let $G$ be a non abelian  subgroup of
$\mathcal{H}(n,\mathbb{R})$. There is a natural action
$\mathcal{H}(n,\mathbb{R})\times \mathbb{R}^{n} :\
\longrightarrow\ \mathbb{R}^{n}$.\ $(f, v)\ \longmapsto\ f(v)$.
For a vector $v \in \mathbb{R}^{n}$, denote by $G(v) := \{f(v):\  f
\in G \}\subset \mathbb{R}^{n}$ the \emph{orbit} of $G$ through
$v$. A subset  $A \subset \mathbb{R}^{n}$  is called
\emph{$G$-invariant} if $f(A)\subset A$  for any $f\in G$; that
is  $A$  is a union of orbits and denote by $\overline{A}$
(resp. $\overset{\circ}{A}$ ) the closure (resp. interior) of $A$.
\\
If  $U$  is an
 open  $G$-invariant set, the orbit  $G(v)\subset U$ \ is
 called \emph{minimal in  $U$} if  $\overline{G(v)}\cap
U = \overline{G(w)}\cap U$  for every  $w\in \overline{G(v)}\cap
U$.
\medskip

 We say that $F$ is an {\it affine subspace} of
$\mathbb{R}^{n}$ with dimension $p$ if  $F=E+a$, for some $a\in
\mathbb{R}^{n}$ and some vector subspace $E$ of $\mathbb{R}^{n}$
with dimension $p$. For every subset $A$ of $\mathbb{R}^{n}$,
denote by $vect(A)$ the vector subspace of $\mathbb{R}^{n}$
generated by all elements of $A$.
\
\\
\\
 Denote by:\
 \\
 -   $id_{\mathbb{R}^{n}}$ the identity map of $\mathbb{R}^{n}$.\
 \\
 - $\Lambda_{G}:=\{\lambda: \ f=(a,\lambda)\in
G\}$.\ It is obvious that $\Lambda_{G}$ is a subgroup of $\mathbb{R}^{*}$ (see Lemma ~\ref{L:00}).
 \\
- $\mathrm{Fix}(f):=\{x\in \mathbb{R}^{n}: \
f(x)=x\}$,  for every $f\in \mathcal{H}(n,
\mathbb{R})$.\
\\
\\
 - $\Gamma_{G}:=\left\{\begin{array}{cc}
  \underset{f\in
G\backslash\mathcal{S}_{n}}{\bigcup}\mathrm{Fix}(f), \ & \ \mathrm{if} \  \
G\backslash\mathcal{S}_{n}\neq\emptyset.\\
\\
 \emptyset, \ \ \ \ \ \ \   & \ \mathrm{if}  \ \
G\subset\mathcal{S}_{n}  \ \ \
\end{array}\right.$
\
\\
\\
- $G_{1}:=G\cap \mathcal{T}_{n}$, \ we have $G_{1}$ \ is a subgroup
of \ $\mathcal{T}_{n}$.
\
\\
- $H_{G}=G_{1}(0)$, we have $H_{G}$ is an additif subgroup of $\mathbb{R}^{n}$.
\
\\
- \ $\gamma_{G}:=\{f(0),\ \ f\in G\cap \mathcal{S}_{n}\}$.
\
\\
- \ $\Omega_{G}:= \Gamma_{G}\cup \gamma_{G}$.
\
\\
- \ $\delta_{G}:=\{f(0),\ \ f\in G\cap (\mathcal{S}_{n}\backslash
\mathcal{T}_{n}) \}$.
\
\\
- \ $E_{G}=Aff(\Omega_{G})$ the smaller affine subspace of \
$\mathbb{R}^{n}$ \ containing \ $\Omega_{G}$.
\\
\\
Remark that $\Omega_{G}\neq\emptyset$, since $\Gamma_{G}\neq\emptyset$ or $\gamma_{G}\neq\emptyset$, and so
$E_{G}\neq\emptyset$.
\

We describe here,  closure of all orbit defined by action of  non
abelian subgroups of $\mathcal{H}(n, \mathbb{R})$. We distinct two
considerable states. When $G\backslash \mathcal{S}_{n}\neq
\emptyset$, this means that  $G$ contains an affine homothety
different to a symmetry (i.e. its homothety ratio $\lambda$ has a
modulus $|\lambda|\neq 1$). In this case we prove that closure of
any orbit is an affine subspace of $\mathbb{R}^{n}$ or it is union
of countable affine subspaces of $\mathbb{R}^{n}$. As consequence,
we deduce that $G$ has a minimal set in $\mathbb{R}^{n}$, which is
contained in closure of all orbit.

In the other state,  $G$ is a non abelian subgroup of
$\mathcal{S}_{n}$, then it contains necessarily  an affine
symmetry $f\in \mathcal{S}_{n}\backslash\mathcal{T}_{n}$. In this
case we prove that closure of any orbit of $G$ is union of at most
two closed subgroups of $\mathbb{R}^{n}$. As consequence, we
deduce that every orbit is minimal in $\mathbb{R}^{n}$.

 For $n=1$, we prove that action for every non
abelian affine group is minimal (i.e. all orbits of $G$ are dense
in $\mathbb{R}$). In \cite{VMS} and \cite{MJVH}, the authors are
interested to the semigroup case, in \cite{MJVH}, Mohamed Javaheri
 has proved a strong density results for the orbits of real
numbers under the action of the semigroup generated by the affine
transformations $T_{0}(x) = \frac{x}{a}$ and $T_{1}(x) = bx +1$,
where $a, b > 1$.
 These density results are formulated as generalizations of the Dirichlet approximation theorem and improve the
results of Bergelson, Misiurewicz, and Senti in \cite{VMS}.

\medskip

 Our principal results can be stated as
follows:

\begin{thm}\label{T:1} Let  $G$ be a non abelian subgroup of
$\mathcal{H}(n,\mathbb{R})$.  One has:\\ $(1)$ If
$G\backslash\mathcal{S}_{n}\neq\emptyset$, then:
\begin{itemize}
\item[(i)]$0\in \overline{\Lambda_{G}}$ and $E_{G}$ is a
$G$-invariant affine subspace of $\mathbb{R}^{n}$ with
dimension $p\geq 1$ .
  \item [(ii)]  $\overline{G(x)}=E_{G}$, for every $x\in E_{G}$.
   \item [(iii)]there exists $a\in E_{G}$ such that  $\overline{G(x)}=\overline{\Lambda_{G}}(x-a)+E_{G}$, for every $x\in U=\mathbb{R}^{n}\backslash E_{G}$.
   \end{itemize}
$(2)$ If  $G\subset{S}_{n}$,  then $H_{G}$ is a closed subgroup
of $\mathbb{R}^{n}$  and  there exists $a\in \mathbb{R}^{n}$ such that
 $\overline{G(x)}=(x+H_{G})\cup(-x+a+ H_{G})$, for
every   $x\in\mathbb{R}^{n}$.
\end{thm}
\medskip

\begin{cor}\label{C:1}Under notations of Theorem ~\ref{T:1}. If
$G\backslash\mathcal{S}_{n}\neq\emptyset$, then:
\begin{itemize}
   \item [(i)] Every orbit in $U$ is minimal in
  $U$.
   \item [(ii)]  $E_{G}$  is a minimal set of \ $G$  in \
$\mathbb{R}^{n}$ \ contained in the closure of every orbit of $G$.
  \item [(iii)] All orbit in $U$ are homeomorphic.
\end{itemize}
\end{cor}
\medskip

 \begin{cor} \label{C:2} Let $G$  be a non abelian subgroup of  $\mathcal{H}(n,
 \mathbb{R})$.  Then:
\begin{itemize}
  \item [(i)] If  $G\backslash \mathcal{S}_{n}\neq \emptyset$,
  then $G$ has no periodic orbit. Moreover, if $G$ is countable
  then it has no closed orbit.
  \item [(ii)] If  $G\subset \mathcal{S}_{n}$,
  then every orbit of $G$ is minimal in $\mathbb{R}^{n}$.
\end{itemize}
\end{cor}
\medskip

 \begin{cor}\label{C:3} Let $G$  be a non abelian subgroup of  $\mathcal{H}(n, \mathbb{R})$
 such that  $G\backslash \mathcal{S}_{n}\neq \emptyset$. Then the following assertions are equivalents:
\begin{itemize}
\item [(1)] $G$ has a dense orbit in in $\mathbb{R}^{n}$.
  \item [(2)] Every orbit
of  $U$  is dense in
$\mathbb{R}^{n}$.
  \item [(3)] $G$ satisfies one of the following:
\begin{itemize}
  \item [(i)] $E_{G}=\mathbb{R}^{n}$
  \item [(ii)] $dim(E_{G})=n-1$ and  $\overline{\Lambda_{G}}=\mathbb{R}$.
    \end{itemize}
\end{itemize}
\end{cor}
\medskip

\begin{rem} \label{r:1} Let  $G$ be a non abelian subgroup of
$\mathcal{H}(n,\mathbb{R})$  such that  $G\backslash
\mathcal{S}_{n}\neq \emptyset$.
\begin{itemize}
  \item [(i)] Suppose that $dim(E_{G})=n-1$ and there exist  $\lambda,\mu\in\Lambda_{G}$  such that $\lambda\mu<0$ and $\frac{log|\lambda|}{log|\mu|}\notin\mathbb{Q}$,
  then $G$ has a dense orbit. (Indeed, in Lemma ~\ref{L:5} we will prove that $\overline{\Lambda_{G}}=\mathbb{R}$ and we apply Corollary
  ~\ref{C:3},(3).(ii)).
  \item [(ii)] If
$\overline{\Lambda_{G}}=\mathbb{R}$  and   $dim(E_{G})<n-1$,  then by Theorem 1.1.(ii),
$G$ has no dense
  orbit and every orbit of  $U$
  is dense in an affine subspace of  $\mathbb{R}^{n}$  with
  dimension  $dim(E_{G})+1$.
  \end{itemize}
\end{rem}
\medskip

\begin{cor}\label{C:4} Let  $G$ be a non abelian subgroup of
$\mathcal{S}_{n}$.  Then  the
following assertions are equivalent:
\begin{itemize}
\item [(i)] $G$ has a dense orbit.
\item [(ii)] Every orbit of  $G$ has a dense orbit.
 \item [(iii)] The orbit  $G(0)$  is dense in    $\mathbb{R}^{n}$.
 \item [(iv)]  $H_{G}$ is dense in  $\mathbb{R}^{n}$.
    \end{itemize}
\end{cor}

\begin{rem}\label{r:2} Let $\mathcal{H}_{n}=\left\{\
f:x\ \longmapsto\ \alpha(x-a)+a;\ a\in\mathbb{R}^{n},\
\alpha\in\mathbb{R}^{*} \right\}$ be the set of all affine
homotheties of $\mathbb{R}^{n}$. Then:
\begin{itemize}
  \item [(i)]  $\mathcal{H}(n,\mathbb{R})=\mathcal{H}_{n}\cup
\mathcal{T}_{n}(\mathbb{R})$.
  \item [(ii)] $\mathcal{H}_{n}$ is not a group. (Indeed;
$\mathcal{H}_{n}\cap
\mathcal{T}_{n}(\mathbb{R})=\{id_{\mathbb{R}^{n}}\}$. \ For
$f=(a,2)$  and $g=\left(2a,\frac{1}{2}\right)$  one has
$f\circ g =T_{a}\in \mathcal{T}_{n}(\mathbb{R})$,  with
$T_{a}:x\longmapsto x+a$.)
  \item [(iii)] There exists a subgroup of $\mathcal{S}_{n}$, having two orbits non homeomorphic. (See example 6.4).
\end{itemize}
\end{rem}
\medskip

For $n=1$, remark that any subgroup of $\mathcal{H}(1,\mathbb{R})$
is a group of affine maps of $\mathbb{R}$. As consequence for
Theorem ~\ref{T:1}, we establish the following  strong result:

\begin{cor}\label{C:5}Let  $G$ be a non abelian group of affine maps of $\mathbb{R}$.
\begin{itemize}
  \item [(i)] If
$G\backslash\mathcal{S}_{1}\neq\emptyset$ then every orbit of $G$
is dense in $\mathbb{R}$.
  \item [(ii)] If $G\subset\mathcal{S}_{1}$ then all orbits
  of $G$ are dense in $\mathbb{R}$ or all orbits are closed and discrete.
  \end{itemize}
\end{cor}
\bigskip

This paper is organized as follows: In Section 2, we introduce some preliminaries Lemmas. Section 3 is devoted to given some results in the case
 $G\backslash \mathcal{S}_{n}\neq\emptyset$. Results in the case when $G$ is a subgroup of $\mathcal{S}_{n}$
are given in Section 4. In Section 5, we prove
Theorem ~\ref{T:1}, Corollaries ~\ref{C:1}, ~\ref{C:2}, ~\ref{C:3}, ~\ref{C:4} and ~\ref{C:5}. In Section 7, we give four examples.
\medskip
\medskip

\bigskip
\section{\textbf{Preliminaries Lemmas}}
\
\\
Recall that  $\mathrm{Fix}(f):=\{x\in \mathbb{R}^{n}: \
f(x)=x\}$,  for every $f\in \mathcal{H}(n,
\mathbb{R})$.\
\\
So $$\mathrm{Fix}(f):=\left\{\begin{array}{cc}
                               \emptyset, & \mathrm{if}\ \ f\in\mathcal{T}_{n}\ \ \ \ \ \ \ \ \ \ \ \ \ \ \ \ \ \ \ \ \ \ \ \\
                               \left\{\frac{a}{2}\right\}, & \mathrm{if} \ \  f=(a,\varepsilon)\in\mathcal{S}_{n}\backslash \mathcal{T}_{n}\ \ \ \ \ \ \  \\
                               \{a\}, &  \ \mathrm{if} \ \ f=(a,\lambda)\in \mathcal{H}(n,
\mathbb{R})\backslash \mathcal{S}_{n}
                             \end{array}
\right.$$\

\begin{lem}\label{L:00} Let $G$ be a non abelian subgroup of $\mathcal{H}(n, \mathbb{R})$.
 The set $\Lambda_{G}$ is a subgroup of $\mathbb{R}^{*}$. Moreover, if $G\backslash \mathcal{S}_{n}\neq\emptyset$,
 then $0\in \overline{\Lambda_{G}}$.
 \end{lem}
\medskip

\begin{proof} Since $id_{\mathbb{R}^{n}}\in G$, so $1\in \Lambda_{G}$. Let $\lambda, \mu\in \Lambda_{G}$
and $f,g\in G$ defined by $f:x\longmapsto \lambda x+a$, and $g:x\longmapsto \mu x+b$,
$x\in \mathbb{R}^{n}$, so $f\circ g^{-1}(x)=f\left(\frac{x}{\mu}-\frac{b}{\mu}\right)=\frac{\lambda}{\mu}x-\frac{\lambda b}{\mu}+a$.
 Hence $\frac{\lambda }{\mu}\in \Lambda_{G}$.\ Moreover, if $G\backslash \mathcal{S}_{n}\neq\emptyset$, $\Gamma_{G}\backslash\{-1,1\}\neq\emptyset$.
 So $\underset{m\to \pm\infty}{lim}\lambda^{m}=0$, for any $\lambda\in\Gamma_{G}$. It follows that $0\in \overline{\Lambda_{G}}$.
  This proves the Lemma.
\end{proof}
\
\\
\begin{lem}\label{L:6}\ \begin{itemize}
\item [(i)]Let $f=(a,\alpha),\ g=(b, \beta)\in\mathcal{H}(n, \mathbb{R})\backslash\mathcal{S}_{n}$ then
$f\circ g=g\circ f$ if and only if $a=b$ or $\alpha=1$ or
$\beta=1$.\
\item [(ii)] If $Fix(f)=Fix(g)$ then $f\circ g=g\circ f$.\
\item [(iii)]  Let $G$ be a non abelian subgroup of $\mathcal{H}(n, \mathbb{R})$ such that $G\backslash \mathcal{S}_{n}\neq\emptyset$, then there exist
$f=(a,\alpha)$, $g=(b, \beta)\in G$ such that $a\neq b$. Moreover,  there exist
 $a,b\in\Gamma_{G}$ such that $a\neq b$.
\end{itemize}
\end{lem}
\medskip

\begin{proof} (i) If $f\circ g(x)=g\circ
f(x)$, for every $x\in \mathbb{R}^{n}$,
\begin{align*}
\mathrm{ then}\ \ \ \ \ \  \ \ \  \lambda(\mu(x-b)+b-a)+a& \
=\mu(\lambda(x-a)+a-b)+b,\\ \mathrm{so}\ \ \ \ \ \  \ \ \  \
-\lambda\mu b+\lambda(b-a)+a& \ =-\mu\lambda a+\mu(a-b)+b,
\end{align*}
 thus
$(a-b)(\lambda\mu-\lambda-\mu+1)=0$. Hence
$(a-b)(\lambda-1)(\mu-1)=0$. This proves the lemma.\
\\
\\
(ii) There are two cases:\
\\
$\bullet$ If $Fix(f)=Fix(g)=\emptyset$ then $f=T_{a}$ and $g=T_{b}$ for some $a,b\in \mathbb{R}^{n}$, so $f\circ g=g\circ f$.
\\
$\bullet$ If $Fix(f)=Fix(g)=a\neq \emptyset$, so $f,g\in \mathcal{H}(n, \mathbb{R})\backslash \mathcal{T}_{n}$, there are four cases:
\\
- $f=(a,\lambda), \ g=(a,\mu)\in \mathcal{H}(n, \mathbb{R})\backslash \mathcal{S}_{n}$, then
$f\circ g(x)=\lambda(\mu(x-a)+a-a)+a=\lambda\mu(x-a)+a=g\circ f(x)$, for every $x\in \mathbb{R}^{n}$.\
\\
-$f=g=(2a,-1)\in \mathcal{S}_{n}\backslash \mathcal{T}_{n}$, so $f\circ g=g\circ f$.\
\\
-$f(2a,-1)\in \mathcal{S}_{n}\backslash \mathcal{T}_{n}$ and $g=(a,\mu)\in \mathcal{H}(n, \mathbb{R})\backslash \mathcal{S}_{n}$, so for every $x\in\mathbb{R}^{n}$,
\begin{align*}
f\circ g(x)& =-(\mu(x-a)+a)+2a  \\
 \mathrm{and} \ \ \ \ \ \ \ \ \    g\circ f(x)& =\mu(-x+2a-a)+a \\
\ \mathrm{so}\ \ \ \ \ \  \ \ \  f\circ g(x)& =g\circ f(x)=-\mu x+(1+\mu)a
\end{align*}
hence \ \ \ $f\circ g=g\circ f$.\
\\
- $f=(a,\lambda)\in \mathcal{H}(n, \mathbb{R})\backslash \mathcal{S}_{n}$ and
$g(2a,-1)\in \mathcal{S}_{n}\backslash \mathcal{T}_{n}$, so as above $f\circ g=g\circ f$. This completes the proof.
\
\\
\\
(iii) Since $G$ is non abelian then the proof of (iii) results from (ii).
\
\\
Moreover, we have $\Gamma_{G}\neq\emptyset$ and let $h=(c,\lambda)\in G\backslash\mathcal{S}_{n}$.
Since $G$ is non abelian then there exists
$h'\in G$ such that $h\circ h'\neq h'\circ h$. By (ii),
   $Fix(h)\neq Fix(h')$, so  $h'(c)\neq c$. Then $h'\circ
   h\circ h'^{-1}=(h'(c), \lambda)\in G\backslash \mathcal{S}_{n}$, hence  $h'(c),c\in\Gamma_{G}$. \
\end{proof}

\medskip

\begin{lem} \label{L:4} Let   $\mathcal{B}=(a_{1},\dots,a_{n})$  be a basis of  $\mathbb{R}^{n}$.
 Then  the smaller affine subspace $\mathcal{A}ff(\mathcal{B})$ of $\mathbb{R}^{n}$  containing
$\{a_{1},\dots,a_{n}\}$ is defined by
 $$\mathcal{A}ff(\mathcal{B}):=\left\{x=\underset{k=1}{\overset{n}{\sum}}\alpha_{k}a_{k}:
  \ \alpha_{k}\in \mathbb{R},\
  \underset{k=1}{\overset{n}{\sum}}\alpha_{k}=1\right\}.$$
\end{lem}
\medskip

\begin{proof} Let $E:=T_{-a_{1}}(\mathcal{A}ff(\mathcal{B}))$. Then $E$ is a vector subspace of $\mathbb{R}^{n}$ \ generated by
$\{a_{2}-a_{1},\dots,a_{n}-a_{1}\}$. Since $E$ is the smaller vector
space containing $0, a_{2}-a_{1},\dots,a_{n}-a_{1},$ so
$T_{a_{1}}(E)$ is the smaller affine subspace of $\mathbb{R}^{n}$
\ containing $\{a_{1},\dots,a_{n}\}$
\end{proof}
\bigskip

\begin{rem}\label{r:3} As consequence of Lemma ~\ref{L:4}, if $E_{G}$ contains $a,a_{1},\dots,a_{n}$  such that $(a_{1},\dots,a_{n})$  is a basis of  $\mathbb{R}^{n}$, and
 $a=\underset{k=1}{\overset{n}{\sum}}\alpha_{k}a_{k}$ with
 $\underset{k=1}{\overset{n}{\sum}}\alpha_{k}\neq 1.$ Then
 $E_{G}=\mathbb{R}^{n}$.
\end{rem}
\medskip

\begin{lem} \label{L:3} Let  $G$ be a non abelian subgroup of
$\mathcal{H}(n,\mathbb{R})$  such that  $G\backslash
\mathcal{S}_{n}\neq \emptyset$.  Then for every
$x\in\mathbb{R}^{n}$  we have $\Gamma_{G}\subset \overline{G(x)}$.
\end{lem}
\medskip

\begin{proof} Let  $x\in\mathbb{R}^{n}$  and
$a\in\Gamma_{G}$.  Since \ $G\backslash \mathcal{S}_{n}\neq
\emptyset$ then there exists  $f=(a,\lambda)\in G\backslash
\mathcal{S}_{n}$, so  $|\lambda|\neq 1$. Suppose that
$|\lambda|>1$ and so
$$\underset{k\longrightarrow-\infty}{lim}f^{k}(x)=\underset{k\longrightarrow-\infty}{lim}\lambda^{k}(x-a)+a=a.$$
\ Hence  $a\in \overline{G(x)}$. It follows that
$\Gamma_{G}\subset \overline{G(x)}$.
\end{proof}
\bigskip

\begin{lem}\label{L:1}Let  $G$ be a non abelian subgroup of
$\mathcal{H}(n,\mathbb{R})$ such that \ $G\backslash
\mathcal{S}_{n}\neq\emptyset$.  Then:
\begin{itemize}
  \item [(i)] if $E_{G}$ is a vector space, there exist
$a_{1},\dots,a_{p}\in \Gamma_{G}$  such that
$(a_{1},\dots,a_{p})$ is a basis of $E_{G}$.
  \item [(ii)] if $G'=T_{-a}\circ G \circ T_{a}$ for some  $a\in\Gamma_{G}$, then $E_{G'}=T_{-a}(E_{G})$ and $\Lambda_{G'}=\Lambda_{G}$.
 \end{itemize}
\end{lem}
\medskip

\begin{proof}(i) Since  $E_{G}$  is a vector
subspace of $\mathbb{R}^{n}$ with dimension $p$, so
 $E_{G}=vect(\Omega_{G})$ where $\Omega_{G}=\Gamma_{G}\cup\gamma_{G}$. As  $G\backslash
\mathcal{S}_{n}\neq\emptyset$ then  $\Gamma_{G}\neq\emptyset$. Let  $a_{1},\dots,a_{k}\in \Gamma_{G}$  and
$b_{k+1},\dots,b_{p}\in \gamma_{G}$  such that
$\mathcal{B}_{1}=(a_{1},\dots,a_{k},b_{k+1},\dots,b_{p})$  is a basis
of  $E_{G}$ and $(a_{1},\dots,a_{k})$ is a basis of $vect(\Gamma_{G})$. For every  $k+1\leq i \leq p$  there exists
$g_{i}\in \mathcal{S}_{n}\cap G$  such that  $g_{i}(0)=b_{i}$.
Since $a_{1}\in \Gamma_{G}$ then there exists  $f\in
G\backslash\mathcal{S}_{n}$  with  $f=(a_{1},\lambda)$. Write
$f_{i}=g_{i}\circ f\circ g^{-1}_{i}$, for every  $k+1\leq i \leq
p$. We have $f_{i}=(g_{i}(a_{1}), \lambda)\in G\backslash
\mathcal{S}_{n}$.  See that
$g_{i}(a_{1})=\varepsilon_{i}a_{1}+b_{i}$, with
$\varepsilon_{i}\in\{-1, 1\}$,  $k+1\leq i \leq p$, so
$g_{i}(a_{1})\in \Gamma_{G}$, for every  $k+1\leq i \leq p$.
\
\\
 \\ Let's show that
$\mathcal{B}_{2}=(a_{1},\dots,a_{k},g_{k+1}(a_{1}),\dots,g_{p}(a_{1}))$
 is a basis of $E_{G}$: \\
 Let $$M=\left[\begin{array}{cc}
            I_{k} & A \\
            0 & I_{p-k}
          \end{array}\right], \ \ \ \mathrm{with} \ \ \ A=\left[\begin{array}{ccc}
                  \varepsilon_{k+1} & \dots  & \varepsilon_{p} \\
                  0 & \dots  & 0 \\
                   \vdots & \ddots  & \vdots  \\
                  0 & \dots   & 0
                \end{array}
\right],$$
\\
 and $I_{k}$, $I_{p-k}$ are respectively the identity matrix
of $M_{k}(\mathbb{R})$ and $M_{p-k}(\mathbb{R})$. So $M$ is invertible and $M(\mathcal{B}_{1})=\mathcal{B}_{2}$. So
$\mathcal{B}_{2}$ is a basis of $E_{G}$ contained in
$\Gamma_{G}$, a contradiction. We conclude that $k=p$.
\
\\
\\
(ii) Second, suppose that  $E_{G}$  is an affine subspace of
 $\mathbb{R}^{n}$  with dimension $p$. Let $a\in \Gamma_{G}$ and
 $G'=T_{-a}\circ G \circ T_{a}$.  Set $f=(a,\lambda)\in G\backslash \mathcal{S}_{n}$, then
 $T_{-a}\circ f\circ T_{a}=(0,\lambda)\in G'\backslash\mathcal{S}_{n}$, so
 $0\in\Gamma_{G'}\subset E_{G'}$, hence $E_{G'}$ is a vector space. By $(i)$ there exists a basis
$(a'_{1},\dots,a'_{p})$  of  $E_{G'}$  contained in
$\Gamma_{G'}$.  Since  $\Gamma_{G'}=T_{-a}(\Gamma_{G})$, we let
 $a_{k}=T_{a}(a'_{k})$,  $1\leq k \leq p$, so  $a_{1},\dots,a_{p}\in\Gamma_{G}$. We have  $\Gamma_{G'}=T_{-a}(\Gamma_{G})\subset T_{-a}(E_{G})$  and  $T_{-a}(E_{G})$  is a
  vector subspace of  $\mathbb{R}^{n}$  with dimension $p$, containing  $a'_{1},\dots,a'_{p}$. So  $E_{G'}=T_{-a}(E_{G})$.
    \\
  See that for every $f=(b,\lambda)\in G\backslash \mathcal{S}_{n}$, $T_{-a}\circ f \circ T_{a}=(b-a,\lambda)$, so $\Lambda_{G'}=\Lambda_{G}$.\
  \end{proof}
\bigskip

\begin{lem} \label{L:1200} Let  $G$\ be a non abelian subgroup of
$\mathcal{H}(n,\mathbb{R})$  such that
$G\backslash\mathcal{S}_{n}\neq\emptyset$ and $E_{G}$ is a vector space. If \ $a_{1},...,a_{p}\in\Gamma_{G}$ \ such that \
$\mathcal{B}_{1}=(a_{1},...,a_{p})$  \ is a bases of \
$E_{G}$ \ then there exists \ $a\in\Gamma_{G}$ \ such
that \ $\mathcal{B}_{2}=(a_{1}-a,...,a_{p}-a)$ \ is also a
basis of \ $E_{G}$.
\end{lem}
\medskip

\begin{proof} Since $a_{p-1}\in\Gamma_{G}$, then there exists $\lambda\in\Lambda_{G}\backslash\{-1 ,1\}$
such that $f_{p-1}=(a_{p-1},\lambda)\in G\backslash \mathcal{S}_{n}$. Let
$a=f_{p-1}(a_{p})=\lambda(a_{p}-a_{p-1})+a_{p-1}$.
 Then $a=\lambda a_{p}+(1-\lambda)a_{p-1}.$ By Lemma ~\ref{L:2}.(ii), $\Gamma_{G}$ is $G$-invariant, so $a\in \Gamma_{G}$.

We have  $P(\mathcal{B}_{1})=\mathcal{B}_{2}$   where
$$P=\left[\begin{array}{cccccc}
            1 & 0 & \dots & \dots & 0 & 0 \\
            0 & \ddots & \ddots & \ddots & \vdots & \vdots \\
            \vdots & \ddots & \ddots & \ddots & \vdots & \vdots \\
            0 & \dots & 0& 1& 0 & 0 \\
            \lambda-1 & \dots & \dots & \lambda-1 & \lambda & \lambda-1\\
            -\lambda & \dots & \dots & -\lambda &\lambda & 1-\lambda
          \end{array}
\right].$$ Since $\mathrm{det}(P)=2\lambda(1-\lambda)\neq 0$  then $P$  is invertible
and so  $\mathcal{B}_{2}$  is a basis of  $\mathbb{R}^{n}$.

\end{proof}
\bigskip

\begin{lem}\label{L:1111} Let $G$ be the subgroup of $\mathcal{H}(n, \mathbb{R})$ generated by $f_{1}=(a_{1},\lambda_{1}),\dots,
 f_{p}=(a_{p},\lambda_{p})\in \mathcal{H}(n, \mathbb{R})\backslash \mathcal{S}_{n}$. Then $E_{G}=\mathcal{A}ff(\{a_{1},\dots, a_{p}\})$.
\end{lem}
\bigskip

\begin{proof}Since $E_{G}=Aff(\Omega_{G})$, it suffices to show that $\Omega_{G}\subset Aff(\{a_{1},\dots, a_{p}\})$.\
\\
\\
(i) First, suppose that $E_{G}$ is a vector space. We will prove that $\Omega_{G}\subset \mathrm{vect}(\{a_{1},\dots, a_{p}\})$:\
\\
Let $f\in G$, so $f=f^{n_{1}}_{i_{1}}\circ\dots\circ f^{n_{q}}_{i_{q}}$ for some $q\in \mathbb{N}^{*}$,
$n_{1},\dots, n_{q}\in\mathbb{Z}$ and $i_{1},\dots, i_{q}\in\{1,\dots, p\}$. So for every $x\in \mathbb{R}^{n}$, \begin{align*}
f(x) & =f^{n_{1}}_{i_{1}}\circ\dots\circ f^{n_{q}}_{i_{q}}(x)\\
\ & =\lambda^{n_{1}}_{i_{1}}(\lambda^{n_{2}}_{i_{2}}(\dots (\lambda^{n_{q-1}}_{i_{q-1}}(\lambda^{n_{q}}_{i_{q}}(x-a_{i_{q}})+a_{i_{q}}-a_{i_{q-1}})+a_{i_{q-1}}\dots)\dots -a_{i_{1}})+a_{i_{1}}\\
\ & = \lambda^{n_{1}}_{i_{1}}\dots\lambda^{n_{q}}_{i_{q}}(x-a_{i_{q}})+\underset{k=1}{\overset{q-1}{\sum}}\lambda^{n_{1}}_{i_{1}}\dots\lambda^{n_{k}}_{i_{k}}(a_{i_{k+1}}-a_{i_{k}})+a_{i_{1}}\\
\ & = \lambda^{n_{1}}_{i_{1}}\dots\lambda^{n_{q}}_{i_{q}} x+\left(\underset{k=1}{\overset{q-1}{\sum}}\lambda^{n_{1}}_{i_{1}}\dots\lambda^{n_{k}}_{i_{k}}a_{i_{k+1}}-\underset{k=1}{\overset{q}{\sum}}\lambda^{n_{1}}_{i_{1}}\dots\lambda^{n_{k}}_{i_{k}}a_{i_{k}}+a_{i_{1}}\right)\\
\ & = \lambda x + a
\end{align*}

where $$(1)\ \ \ \ \ \ \left\{\begin{array}{c}
                \lambda=\lambda^{n_{1}}_{i_{1}}\dots\lambda^{n_{q}}_{i_{q}} \ \ \ \ \ \ \ \ \ \ \ \ \ \ \  \ \ \ \ \ \ \ \ \ \ \ \  \ \ \ \ \ \ \ \ \ \ \ \ \ \ \ \ \ \\
                a=\underset{k=1}{\overset{q-1}{\sum}}\lambda^{n_{1}}_{i_{1}}\dots\lambda^{n_{k}}_{i_{k}}a_{i_{k+1}}-\underset{k=1}{\overset{q}{\sum}}\lambda^{n_{1}}_{i_{1}}\dots\lambda^{n_{k}}_{i_{k}}a_{i_{k}}+a_{i_{1}}
              \end{array}\right.$$
-\ If $|\lambda|\neq 1$, then $f(x)= \lambda x + a=\lambda \left(x-\frac{a}{1-\lambda}\right) + \frac{a}{1-\lambda}$, $x\in \mathbb{R}^{n}$, so $f=\left(\frac{a}{1-\lambda},\ \lambda\right)$, with $\frac{a}{1-\lambda}\in \mathrm{vect}(\{a_{1},\dots, a_{p}\})$, so $\Gamma_{G}\subset \mathrm{vect}(a_{1},\dots,a_{p})$.\
\\
-\ If $|\lambda|= 1$, then $f=\left(a,\ \frac{\lambda}{|\lambda|}\right)$, with $a=f(0)\in \mathrm{vect}(\{a_{1},\dots, a_{p}\})$, so $\gamma_{G}\subset \mathrm{vect}(a_{1},\dots,a_{p})$.\
\\
It follows that $\Omega_{G}\subset \mathrm{vect}(\{a_{1},\dots, a_{p}\})$.\
\\
\\
(ii) Second, suppose that $E_{G}$ is an affine space. We will prove that $\Omega_{G}\subset \mathcal{A}ff(\{a_{1},\dots, a_{p}\})$:
Let $G'=T_{-a_{1}}\circ G \circ T_{a_{1}}$, so $G'$ is generated by
$f'_{k}=T_{-a_{1}}\circ f_{k} \circ T_{a_{1}}=(a_{k}-a_{1},\lambda_{k})$, $k=1,\dots, p$. By Lemma ~\ref{L:5}.(ii), $E_{G'}=T_{-a_{1}}(E_{G})$ is a vector space and by (i) we have
$E_{G'}\subset vect(\{a_{2}-a_{1},\dots, a_{p}-a_{1}\})$,  so $E_{G}=T_{a_{1}}(E_{G'})\subset T_{a_{1}}(\mathrm{vect}(\{a_{2}-a_{1},\dots, a_{p}-a_{1}\}))
=\mathcal{A}ff(\{a_{1},a_{2},\dots, a_{p}\})$, so the proof is complete.
\end{proof}

\section{\textbf{Some results in the case  $G\backslash \mathcal{S}_{n}\neq\emptyset$}}

In this case, $G$ contains an affine homothety having a ratio with module different to $1$. In
the following, we give some Lemmas and propositions, will be
used to prove Theorem ~\ref{T:1}.

\medskip

\begin{lem}\label{L:2} Let $G$ be a non abelian subgroup of $\mathcal{H}(n,\mathbb{R})$, then:
\begin{itemize}
  \item [(i)] If $E_{G}$ is a vector
subspace of $\mathbb{R}^{n}$, then $\{f(0), f\in G\}\subset E_{G}$.
  \item [(ii)] If \ $\Gamma_{G}\neq\emptyset$,  then $\Gamma_{G}$ and $E_{G}$ are  $G$-invariant.
\end{itemize}
\end{lem}
\medskip

\begin{proof}(i) Let $f\in G$, there are two cases:
\
\\
- If $f\in G\backslash \mathcal{S}_{n}$, then $f=(a,\lambda)$, for
some $\lambda\in\Lambda_{G}$  and \ $a\in\Gamma_{G}\subset E_{G}$. Therefore \
$f(x)=\lambda(x-a)+a$, $x\in \mathbb{R}^{n}$  and
$f(0)=(1-\lambda)a$, so $f(0)\in E_{G}$ since $E_{G}$ is a vector space.
\
\\
- If  $f\in \mathcal{S}_{n}$, then $f=(a,\varepsilon)$, with $|\varepsilon|=1$ and so
$f(0)=a\in E_{G}$.
\
\\
\\
(ii) Suppose that $\Gamma_{G}\neq\emptyset$:\
\\
$\Gamma_{G}$ is $G$-invariant: Let  $a\in \Gamma_{G}$ and \
$g\in G$  then there exists
$\lambda\in\mathbb{R}\backslash\{-1,1\}$  such that
$f=(a,\lambda)\in G\backslash S_{n}$.  We let \ $h=g\circ f\circ
g^{-1}\in G$. We obtain  $h=(a',\lambda)\in G\backslash S_{n}$
with  $a'=g(a)$. It follows that  $g(a)\in \Gamma_{G}$ \ and so
$\Gamma_{G}$  is $G$-invariant.
\
\\
 $E_{G}$  is $G$-invariant: Let $a\in \Gamma_{G}$
and $G'=T_{-a}\circ G \circ T_{a}$. We have  $G'$ is a non abelian
subgroup of  $\mathcal{H}(n, \mathbb{R})$  and
$E_{G'}=T_{-a}(E_{G})$  is a vector subspace of
$\mathbb{R}^{n}$. Let  $f\in G'$ having the form $f(x)=\lambda
x+b$, $x\in \mathbb{R}^{n}$.  By (i), $b=f(0)\in\Gamma_{G'}
\subset E_{G'}$. So for every  $x\in E_{G'}$, $f(x)\in
E_{G'}$, hence $E_{G'}$ is $G'$-invariant. By Lemma
~\ref{L:1}.(ii) one has  $E_{G}=T_{-a}(E_{G'})$,  so it is
$G$-invariant.
\end{proof}
\bigskip

\begin{lem}\label{L:7} Let $\lambda>1$  and  $H^{\lambda}:=\{q\lambda^{p}(1-\lambda^{p}), \ \ p,q\in \mathbb{Z}\}$.
Then \ $H^{\lambda}$  is dense in  $\mathbb{R}$.
\end{lem}
\medskip

\begin{proof}  Let  $x,y\in \mathbb{R}^{*}_{+}$,   such that  $x<y$.  Since
$\underset{p\longrightarrow-\infty}{lim}\frac{y-x}{\lambda^{p}(1-\lambda^{p})}=+\infty$
\ then there exists  $p\in\mathbb{Z}^{*}_{-}$  such that
$\frac{y-x}{\lambda^{p}(1-\lambda^{p})}>1$.  Therefore there
exists $q\in \mathbb{Z}$ such that
$\frac{x}{\lambda^{p}(1-\lambda^{p})}<q<\frac{y}{\lambda^{p}(1-\lambda^{p})}$.
 Since $\lambda>1$ and $p\neq0$  then   $1-\lambda^{p}>0$   and so   $x<
 q\lambda^{p}(1-\lambda^{p})<y$ and  $-y<-
 q\lambda^{p}(1-\lambda^{p})<-x$.   Hence  $\mathbb{R}^{*}_{+}\subset
 \overline{H^{\lambda}}$  and  $\mathbb{R}^{*}_{-}\subset \overline{H^{\lambda}}$. It follows that
  $H^{\lambda}$ is dense in  $\mathbb{R}$.
\end{proof}
\medskip

\begin{lem}\label{L:8} Let $\lambda>1$, $a\in \mathbb{R}^{n}\backslash \{0\}$
  and $H^{\lambda}_{a}:=\{q\lambda^{p}(1-\lambda^{p})a+a, \ \ p,q\in
\mathbb{Z}\}$. If $G$ is the group generated by  $f=(a, \lambda)\in \mathcal{H}(n,
\mathbb{R})\backslash \mathcal{S}_{n}$  and
 $h=\lambda. id_{\mathbb{R}^{n}}$, then
$H^{\lambda}_{a}\subset G(a)\subset\mathbb{R}a$. Moreover, $T_{(1-\lambda^{m})a}\in G$, for every $m\in \mathbb{Z}$.
\end{lem}
\medskip

\begin{proof}
Let  $p,q\in \mathbb{Z}$.  For every  $z\in \mathbb{R}^{n}$ we
have \ \begin{align*}
           h^{-p}\circ
f^{p}(z) &
=\lambda^{-p}(\lambda^{p}(z-a)+a)=z+(\lambda^{-p}-1)a,\ \ \ \ \ \ \ \ (1)\\
\mathrm{so} \ \ \ \   \left(h^{-p}\circ f^{p}\right)^{q-1}(z)& \
=z+ (q-1)(\lambda^{-p}-1)a.\\
\\
   \mathrm{For} \ \ z=h^{-p}(a) \
\ \mathrm{we \ have} &\ \\
  \ \ \left(h^{-p}\circ
f^{p}\right)^{q-1}(h^{-p}(a))& =\lambda^{-p}a+
(q-1)(\lambda^{-p}-1)a=q\lambda^{-p}a-(q-1)a.
\end{align*}
\begin{align*}
\mathrm{Then}\ \ \ \ \ \ \ \ \ \ \ \ \ \ \ \ \ \ \ \ \ \ \ \ \ \ \ \ \ \ \ &\ \\
 f^{2p}\circ \left(h^{-p}\circ
f^{p}\right)^{q-1}\circ
h^{-p}(a)& \ =\lambda^{2p}\left(q\lambda^{-p}a-(q-1)a-a\right)+a.\\
\\
\ & \  =q\lambda^{p}(1-\lambda^{p})a+a
 \end{align*}
 \\
It follows that   $q\lambda^{p}(1-\lambda^{p})a+a\in G(a)$  and
so  $H^{\lambda}_{a}\subset G(a)$. Since $h(\mathbb{R}a)=\mathbb{R}a$ and $f(\alpha a)=\lambda(\alpha
a-a)+a=(\lambda\alpha-\lambda+1)a$ then $\mathbb{R}a$ is
$G$-invariant, so $G(a)\subset \mathbb{R}a$.\
\\
Moreover, by taking $p=-m$ \ in (1), for some $m\in \mathbb{Z}$, we obtain $$h^{m}\circ f^{m}(z)=z+(\lambda^{m}-1)a,
 \ \ z\in \mathbb{R}^{n}.$$ Then
 $T_{(\lambda^{m}-1)a}\in G$ and so $T_{(1-\lambda^{m})a}=T^{-1}_{(\lambda^{m}-1)a}\in G$. The proof is complete.
\end{proof}
\bigskip

\begin{lem}\label{L:9} Let  $\lambda>1$,   $a,b\in \mathbb{R}^{n}$  with $a\neq b$.  If  $G$  is the group generated
by $f=(a, \lambda)$  and  $g=(b,\lambda)$ then
$\overline{G(a)}=\mathbb{R}(b-a)+a$.
\end{lem}
\medskip

\begin{proof} Let  $\lambda>1$,   $a,b\in \mathbb{R}^{n}$  with $a\neq
b$  and  $G$ \ be the group generated by\ $f=(a, \lambda)$  and
 $g=(b,\lambda)$ . Denote by $G'=T_{-b}\circ G \circ T_{b}$, then
$G'$ is a subgroup of \ $\mathcal{H}(n, \mathbb{R})$  and it is
generated by  $h=T_{-b}\circ f \circ T_{b}$ and
$g'=T_{-b}\circ g \circ T_{b}$.  We obtain  $h=\lambda.id_{\mathbb{R}^{n}}$ and $g'=(a-b, \lambda)$.

Since $\lambda>1$, $g'\in G'\backslash \mathcal{S}_{n}$, so by Lemma ~\ref{L:8}, we have  $H^{\lambda}_{a-b}\subset G'(a-b)$, where
 $H^{\lambda}_{a-b}:=\{q\lambda^{p}(1-\lambda^{p})(a-b)+a-b,  \
p,q\in \mathbb{Z}\}$. Since $a-b\neq 0$  then
$H^{\lambda}_{a-b}$  and  $H^{\lambda}:=\{q\lambda^{p}(1-\lambda^{p}),  \
p,q\in \mathbb{Z}\}$  are  homeomorphic. By Lemma ~\ref{L:7}, we have
$H^{\lambda}$ is dense in  $\mathbb{R}$  so  $H^{\lambda}_{a-b}$ is dense
in  $\mathbb{R}(a-b)$. Since  $H^{\lambda}_{a-b}\subset G'(a-b)$
 and by Lemma ~\ref{L:8}, $G'(a-b)\subset \mathbb{R}(a-b)$, so
$\overline{G'(a-b)}=\mathbb{R}(a-b)$. We conclude that
$\overline{G(a)}=T_{b}(\mathbb{R}(a-b))=\mathbb{R}(a-b)+b$.
As $\mathbb{R}(a-b)+b=\mathbb{R}(a-b)+(a-b)+b=\mathbb{R}(b-a)+a$. It follows that $\overline{G(a)}=\mathbb{R}(b-a)+a$.
\end{proof}
\bigskip

\begin{lem}\label{L:10} Let  $\lambda>1$, $\mu\in \mathbb{R}\backslash\{0,1\}$   $a,b\in \mathbb{R}^{n}$  with $a\neq b$.
 If  $G$  is the group generated by $f=(a, \lambda)$  and  $g=(b,\mu)$  then
$\overline{G(a)}=\mathbb{R}(b-a)+a$.
\end{lem}
\medskip

\begin{proof} Let  $\lambda>1$, $\mu\in \mathbb{R}\backslash\{0,1\}$   $a,b\in \mathbb{R}^{n}$  with $a\neq b$ and  $G$
 is the group generated by $f=(a, \lambda)$  and  $g=(b,\mu)$, then by Lemma ~\ref{L:6}.(i), $G$
is non abelian.
\
\\
(i) First, we will show that $\mathbb{R}(b-a)+a$ is $G$-invariant:

Let $\alpha\in \mathbb{R}$, and  $x=\alpha(b-a)+a$  we have

\begin{align*}
f(x)& =\lambda(\alpha(b-a)+a-a)+a\\
\ & =\lambda\alpha(b-a)+a\\
\end{align*}
and
\begin{align*}
f(x)& =\mu(\alpha(b-a)+a-b)+b\\
\  & =\mu(\alpha-1)(b-a)+b-a+a.\\
\ & =(1+\mu(\alpha-1))(b-a)+a
\end{align*}

  So  $f(x),\ g(x)\in \mathbb{R}(b-a)+a$.
\
\\
\\
(ii) Second, we let $g'=g\circ f \circ g^{-1}$, we have $g'=(g(a),
\lambda)\in G\backslash \mathcal{S}_{n}$. Since $a\neq b$ and $\mu\neq 1$ then
$$g(a)-a=\mu(a-b)+b-a=(1-\mu)(b-a)\neq 0.$$ If $G'$ is the subgroup
of $G$ generated by $f$ and $g'$ then by Lemma ~\ref{L:9} we have
$\overline{G'(a)}=\mathbb{R}(g(a)-a)+a$.   Since
$g(a)=\mu(a-b)+b$  then
\begin{align*}
\mathbb{R}(g(a)-a)+a& =\mathbb{R}(\mu(a-b)+b-a)+a\\
\ &\ =\mathbb{R}(1-\mu)(b-a)+a\\
\ & =\mathbb{R}(b-a)+a.
\end{align*}  By (i), we have
$\mathbb{R}(b-a)+a$  is  $G$-invariant so $\overline{G'(a)}\subset \overline{G(a)}\subset \mathbb{R}(b-a)+a$, hence
$\overline{G(a)}=\mathbb{R}(b-a)+a$.
\end{proof}
\bigskip

\begin{prop}\label{p:1} Let  $G$\ be a non abelian subgroup of
$\mathcal{H}(n,\mathbb{R})$  such that
$G\backslash\mathcal{S}_{n}\neq\emptyset$.  Then for every $x\in
E_{G}$, we have  $\overline{G(x)}=E_{G}$.
\end{prop}
\medskip

To prove the above Proposition, we need the following Lemmas:

\begin{lem}\label{L:11} Let  $G$ be a non abelian subgroup of
$\mathcal{H}(n,\mathbb{R})$.  Let  $f\in G$, $u,v\in\mathbb{R}^{n}$  then
$f(\mathbb{R}u+v)=\mathbb{R}u+f(v)$.
\end{lem}
\medskip

\begin{proof} Every  $f\in G$  has the form
$f(x)=\lambda x+a$, $x\in \mathbb{R}^{n}$. Let $\alpha\in
\mathbb{R}$  then $f(\alpha u+v)=\lambda(\alpha
u+v)+a=\lambda\alpha u+(\lambda u+v)=\lambda\alpha u+f(v)$. So $f(\mathbb{R}u+v)=\mathbb{R}u+f(v)$.
\end{proof}
\medskip

\begin{lem}\label{L:12} Let  $G$ be a non abelian subgroup of
$\mathcal{H}(n,\mathbb{R})$  such that $E_{G}$ is a vector
subspace of $\mathbb{R}^{n}$ and $\Gamma_{G}\neq\emptyset$.  Let \ $a, a_{1},\dots,a_{p}\in
\Gamma_{G}$  such that  $(a_{1},\dots,a_{p})$ and $(a_{1}-a,\dots,a_{p}-a)$ are two basis of
$E_{G}$ and let $D_{k}=\mathbb{R}(a_{k}-a)+a$,
$1\leq k \leq p$.  If \ $D_{k}\subset\overline{G(a)}$  for every
 $1\leq k \leq p$,  then  $\overline{G(a)}=E_{G}$.
\end{lem}
\medskip

\begin{proof} The proof is done by induction on  $\mathrm{dim}(E_{G})=p\geq 1$.
\\
For $p=1$,  by Lemma ~\ref{L:6}.(iii) there exist
$a,b\in\Gamma_{G}$  with  $a\neq b$, since $G$ is non abelian and $\Gamma_{G}\neq\emptyset$.  In this case
$D_{1}=\mathbb{R}(b-a)+a=\mathbb{R}=E_{G}$,  then if $D_{1}\subset \overline{G(a)}$ so  $\overline{G(a)}=E_{G}$.
$\overline{\Gamma_{G}}=E_{G}$.
\
\\
Suppose that Lemma ~\ref{L:12} is true until dimension $p-1$. Let  $G$ be
a non abelian subgroup of \ $\mathcal{H}(n,\mathbb{R})$ with $\Gamma_{G}\neq\emptyset$ and let
$a, a_{1},\dots,a_{p}\in \Gamma_{G}$  such that
$(a_{1},\dots,a_{p})$  is a basis of  $E_{G}$.  Suppose that
$D_{k}\subset\overline{G(a)}$  for every  $1\leq k \leq p$.
\\
Denote by $H$  the vector
subspace of  $E_{G}$ generated by
$(a_{1}-a),\dots,(a_{p-1}-a)$ and $\Delta_{p-1}=T_{a}(H)$. We have  $\Delta_{p-1}$  is an
affine subspace of $E_{G}$  and it contains
$a,a_{1},\dots,a_{p-1}$.
\\
Set $\lambda,\lambda_{k}\in \Gamma_{G}$,   $1\leq k
\leq p-1$   such that $f=(a,\lambda), f_{k}=(a_{k},\lambda_{k})\in G\backslash\mathcal{S}_{n}$. Suppose that  $\lambda>1$  and
 $\lambda_{k}>1$, for every  $1\leq k \leq p-1$ (leaving to replace $f$ and $f_{k}$ respectively by
$f^{2}$ or $f^{-2}$ and  by $f^{2}_{k}$ or $f^{-2}_{k}$). Let
$G_{k}$ be the group generated by $f$ and $f_{k}$, $1\leq k \leq
p-1$. By Lemma ~\ref{L:10} we have $\overline{G_{k}(a)}=D_{k}$. Let  $G'$
be the subgroup of $G$ generated by $f$, $f_{1}$,\dots,$f_{p-1}$,   then
$D_{k}\subset \overline{G'(a)}$ for every $1\leq k \leq p-1$.
\\
By Lemma ~\ref{L:1111} we have $E_{G'}=\Delta_{p-1}$. Let
$G''=T_{-a}\circ G'\circ T_{a}$,   by Lemma ~\ref{L:1}.(ii) we have
$E_{G''}=T_{-a}(\Delta_{p-1})=H$  and
$D'_{k}=T_{-a}(D_{k})\subset \overline{G''(0)}$  for every $1\leq
k \leq p-1$. By induction hypothesis applied to $G''$ we have
$\overline{G''(0)}=H$  so  $\overline{G'(a)}=\Delta_{p-1}$.
Since $G'(a)\subset G(a)$, then $$\Delta_{p-1}\subset \overline{G(a)}\ \ \ \ \ \ \ \ \ \
(1).$$\
\\
Let  $x\in E_{G}\backslash \Delta_{p-1}$  and
$D=\mathbb{R}(a_{p}-a)+x$. Since $(a_{1}-a,\dots,a_{p}-a)$ is a basis of $E_{G}$, so
$H\oplus\mathbb{R}(a_{p}-a)=E_{G}$ with $x,a\in E_{G}$,
then $x-a=z+\alpha(a_{p}-a)$ with $z\in H$ and
$\alpha\in\mathbb{R}$. Let $y=z+a$, as  $H+a=\Delta_{p-1}$ we
have  $y\in\Delta_{p-1}$, and
\begin{align*}
y& =-\alpha(a_{p}-a)+x\in D\\
\ & =z+a\\
\ & =x-a-\alpha(a_{p}-a)+a\\
\ & =-\alpha(a_{p}-a)+x\in D
\end{align*}

Hence \ \ $y\in \Delta_{p-1}\cap D.$
\
\\
\\
 By (1) we have $y\in
\overline{G(a)}$.  Then there exists a sequence
$(f_{m})_{m\in\mathbb{N}}$  in  $G$  such that \
$\underset{m\longrightarrow+\infty}{lim}f_{m}(a)=y.$ For every
$m\in\mathbb{N}$  denote by  $f_{m}=(b_{m}, \lambda_{m})$.
\
\\
\\
By Lemma ~\ref{L:11} we have
$f_{m}(D_{p})=f_{m}(\mathbb{R}(a_{p}-a)+a)=\mathbb{R}(a_{p}-a)+f_{m}(a)$.
Since  $\underset{m\longrightarrow+\infty}{lim}f_{m}(a)=y$  then
\
$$\underset{m\longrightarrow+\infty}{lim}f_{m}(D_{p})=\mathbb{R}(a_{p}-a)+y$$\
\\
 As  $y\in D$  then  $y-x\in
\mathbb{R}(a_{p}-a)$,  thus
$\mathbb{R}(a_{p}-a)+y=\mathbb{R}(a_{p}-a)+x=D$  and so
$\underset{m\longrightarrow+\infty}{lim}f_{m}(D_{p})=D$.\
\\
Since $D_{p}\subset \overline{G(a)}$  then \ $D\subset
\overline{G(a)}$, so  $x\in \overline{G(a)}$, hence
$E_{G}\backslash \Delta_{p-1}\subset \overline{G(a)}$. By
$(1)$ we obtain  $E_{G}\subset \overline{G(a)}$. Since $\Gamma_{G}\neq\emptyset$ then by Lemma
~\ref{L:2}.(ii), we have $E_{G}$  is $G$-invariant, so $G(a)\subset E_{G}$ since $a\in
E_{G}$. It follows that
$\overline{G(a)}=E_{G}$.
\end{proof}
\bigskip

\begin{proof}[Proof of Proposition ~\ref{p:1}] Let $G$ be a non abelian
subgroup of $\mathcal{H}(n, \mathcal{R})$. Since $G\backslash \mathcal{S}_{n}\neq\emptyset$ then $\Gamma_{G}\neq\emptyset$
 and suppose that
$E_{G}$  is a vector subspace of $\mathbb{R}^{n}$, (one can replace $G$ by  $G'=T_{-a}\circ G \circ T_{a}$, for some $a\in \Gamma_{G}$).

First, we will prove that there exists $a\in\Gamma_{G}$ such that
 $\overline{G(a)}=E_{G}$.  By Lemmas ~\ref{L:1},(i) and ~\ref{L:1200}, there exists
$a,a_{1},\dots,a_{p}\in\Gamma_{G}$  such that  $(a_{1},\dots,a_{p})$ and
 $(a_{1}-a,\dots,a_{p}-a)$ are two basis of $E_{G}$. Denote by $D_{k}=\mathbb{R}(a_{k}-a)+a$, $1\leq k \leq p$.
Since $a\in \Gamma_{G}$, then there exists $f\in G$ such that
$f=(a,\lambda)$. Suppose that $\lambda>1$ (one can replace $f$ by $f^{2}$
or $f^{-2}$). By Lemma ~\ref{L:10},  $D_{k}\subset
\overline{G(a)}$, for every $1\leq k \leq p$. By Lemma ~\ref{L:12}, we have
$\overline{G(a)}=E_{G}$.
\medskip

Second, let $x\in E_{G}$ and by
Lemma ~\ref{L:3} we have $\Gamma_{G}\subset \overline{G(x)}$ and by Lemma
~\ref{L:2}.(ii), $\Gamma_{G}$ is $G$-invariant. Since  $a\in \Gamma_{G}$
 then $$E_{G}=\overline{G(a)}\subset
\overline{\Gamma_{G}}\subset \overline{G(x)}.$$
It follows that $\overline{G(x)}=E_{G}$ since $E_{G}$ is $G$-invariant (Lemma ~\ref{L:2},(ii)). The proof is complete.
\end{proof}
\bigskip

\begin{prop}\label{p:2} Let  $G$ be a non abelian subgroup of
$\mathcal{H}(n,\mathbb{R})$.  Suppose that
  $G\backslash\mathcal{S}_{n}\neq\emptyset$ and $E_{G}$ is a vector space.  Then
  for every $x\in \mathbb{R}^{n}\backslash E_{G}$,  we have
   $\overline{G(x)}=\overline{\Lambda_{G}}.x+E_{G}.$
\end{prop}
\medskip

To prove the above Proposition, we need the following Lemma:
\begin{lem}\label{L:13} Let  $G$ be a non abelian subgroup of
$\mathcal{H}(n,\mathbb{R})$  such that
$G\backslash\mathcal{S}_{n}\neq\emptyset$. For every
$\lambda\in \Lambda_{G}\backslash\{-1,1\}$ and for every \ $b\in E_{G}$,
there exists a sequence $(f_{m})_{m\in\mathbb{N}}$ in $G$ such
that  $\underset{m\longrightarrow+\infty}{lim}f_{m}=f$, with
$f=(b,\lambda)\in \mathcal{H}(n, \mathbb{R})\backslash\mathcal{S}_{n}$.
\end{lem}
\bigskip

\begin{proof} Let  $\lambda\in
\Lambda_{G}\backslash\{-1,1\}$ and  $b\in E_{G}$.   Given  $g=(a,\lambda)\in
G\backslash \mathcal{S}_{n}$, so $a\in\Gamma_{G}\subset E_{G}$. By Proposition ~\ref{p:1}, we have
$\overline{G(a)}=E_{G}$.  Then there exists a sequence
$(g_{m})_{m\in\mathbb{N}}$ in $G$ such that
$\underset{m\longrightarrow+\infty}{lim}g_{m}(a)=b$. For every
$m\in \mathbb{N}$, denote by  $f_{m}=g_{m}\circ g \circ
g^{-1}_{m}$, so  $f_{m}=(g_{m}(a),\lambda)$. Hence
$\underset{m\longrightarrow+\infty}{lim}f_{m}=f$, with
$f=(b,\lambda)$.
\end{proof}
\bigskip

\begin{lem}\label{L:130} Let  $G$ be a non abelian subgroup of
$\mathcal{H}(n,\mathbb{R})$  such that $E_{G}$ is a vector space and
$G\backslash\mathcal{S}_{n}\neq\emptyset$. Then:\
\\
(i)  For every
$b\in E_{G}$ there exists a sequence $(T_{b_{m}})_{m\in\mathbb{Z}}$ in $G\cap \mathcal{T}_{n}$ such
that  $\underset{m\to-\infty}{lim}T_{b_{m}}=T_{b}$.\
\\
(ii) If $-1\in\Lambda_{G}$, then for every
$b\in E_{G}$ there exists a sequence $(S_{m})_{m\in\mathbb{N}}$ in $G\cap (\mathcal{S}_{n}\backslash \mathcal{T}_{n})$ such
that  $\underset{m\longrightarrow+\infty}{lim}S_{m}=S=(b,-1)$.
\end{lem}
\bigskip

\begin{proof}(i) $\bullet$ First, suppose that $b\in\Gamma_{G}$, then there exists $f\in G\backslash \mathcal{S}_{n}$
with $f=(b,\lambda)$. Since $G$ is non abelian set $g\in G$ such that
$f\circ g \neq g\circ f$.  Set $h=g\circ f \circ g^{-1}$, so $h=(g(b),\lambda)$. Let $G'=T_{-g(b)}\circ G\circ T_{g(b)}$, $f'=T_{-g(b)}\circ f\circ T_{g(b)}$ and
 $h'=T_{-g(b)}\circ h\circ T_{g(b)}$, so $f'=(b-g(b), \lambda)$ and $h'=(0,\lambda)=\lambda id_{\mathbb{R}^{n}}$. By Lemma ~\ref{L:9}, for every
 $m\in \mathbb{Z}$, $T'_{m}=T_{(1-\lambda^{m})(b-g(b))}\in G'$. Write $T_{b_{m}}=T_{g(b)}\circ T'_{m}\circ T_{-g(b)}$,
 so $$b_{m}=(1-\lambda^{m})(b-g(b))+g(b)=(1-\lambda^{m})b+\lambda^{m}g(b), \ \ m\in\mathbb{Z}.$$

 Since $|\lambda|\neq1$, suppose that $\lambda>1$, so $\underset{m\to -\infty}{lim}(1-\lambda^{m})a+\lambda^{m}g(b)=b$.
 It follows that the sequence $(T_{b_{m}})_{m}\in G\cap \mathcal{T}_{n}$ and $\underset{m\to-\infty}{lim}T_{b_{m}}=T_{b}$.\
 \\
  $\bullet$ Now, suppose that $b\in E_{G}$ and let $a\in\Gamma_{G}$. By Proposition~\ref{p:1}, $\overline{G(a)}=E_{G}$, so
  there exists a sequence $(g_{k})_{k}$ in $G$ such that  $\underset{k\longrightarrow+\infty}{lim}g_{k}(a)=b$. By above
  state, there exists a sequence $(T_{a_{m}})_{m}\in G\cap \mathcal{T}_{n}$ such that $\underset{m\to-\infty}{lim}T_{a_{m}}=T_{a}$. Set
  $T_{b_{m,k}}=g_{k}\circ T_{a_{m}}\circ g_{k}^{-1}$, one has $T_{b_{m,k}}=(g_{k}(a_{m}), 1)\in G\cap \mathcal{T}_{n}$. We have
  $\underset{m\to-\infty}{lim}b_{m,k}=\underset{m\to-\infty}{lim}g_{k}(a_{m})=g_{k}(a)$, so $$\underset{_{\|(k,-m)\|\to+\infty}}{lim}b_{m,k}=\underset{k\to+\infty}{lim}g_{k}(a)=b.$$
   So $\underset{_{\|(k,-m)\|\to+\infty}}{lim}T_{b_{m,k}}=T_{b}$. This complete the proof of (i).\
   \\
   \\
   (ii) Suppose that $-1\in\Lambda_{G}$ and let $b\in E_{G}$. Then there exists $f=(a,-1)\in G\cap \mathcal{S}_{n}$. By Lemma~\ref{L:2}.(i),
    $a=f(0)\in E_{G}$, so $b-a\in E_{G}$, since $E_{G}$ is a vector space. By (i),
   there exists a sequence $(T_{a_{m}})_{m\in\mathbb{Z}}$ in $G\cap \mathcal{T}_{n}$ such
that  $\underset{m\to-\infty}{lim}T_{a_{m}}=T_{b-a}$. Set $S_{m}=T_{a_{m}}\circ f$. We have $S_{m}=(a+a_{m},-1)\in G\cap \mathcal{S}_{n}$.
Since $\underset{m\to +\infty}{lim}a_{m}=b-a$, then $\underset{m\to +\infty}{lim}a+a_{m}=b$, so $\underset{m\to +\infty}{lim}S_{m}=S=(b,-1)$. The proof is complete.
\end{proof}
\bigskip

\begin{proof}[Proof of Proposition ~\ref{p:2}] Let  $G$ be a non abelian subgroup of
$\mathcal{H}(n,\mathbb{R})$  such that
$G\backslash\mathcal{S}_{n}\neq\emptyset$ and $E_{G}$ is a vector space. Let $x\in U=\mathbb{R}^{n}\backslash E_{G}$.\
\\
Lest' prove that $\overline{\Lambda_{G}}.x+E_{G}\subset\overline{G(x)}$: \ Let $\alpha\in\Lambda_{G}$
and  $a\in E_{G}$.
\\
$\bullet$ Suppose that $\alpha\in\Lambda_{G}\backslash\{-1,1\}$. Since $E_{G}$ is a vector space, $a'=\frac{a}{1-\alpha}\in E_{G}$.
 By Lemma ~\ref{L:13} there exists a sequence $(f_{m})_{m}$ in $G$ such that
  $\underset{m\longrightarrow +\infty}{lim}f_{m}=f=(a', \alpha)\in G\backslash \mathcal{S}_{n}$. Then
  \begin{align*}
  f(x)& =\alpha (x-a')+a'\\
  \ &\ =\alpha x+ (1-\alpha)a'\\
  \ & =\alpha x+ a\in \overline{G(x)},
  \end{align*}
    so  $$\left(\Lambda_{G}\backslash\{1,1\}\right).x+E_{G}\subset \overline{G(x)}.$$\
    \\
    $\bullet$ Suppose that $\alpha\in\Lambda_{G}\cap\{-1,1\}$.\
    \\
    - If $\alpha=1$, by Lemma ~\ref{L:130}.(i), there exists a sequence $(T_{a_{m}})_{m}$ in $G$ such that
  $\underset{m\longrightarrow +\infty}{lim}T_{a_{m}}=T_{a}$. So $T_{a}(x)=x+a\in \overline{G(x)}$.\
  \\
  - If $\alpha=1$, by Lemma ~\ref{L:130}.(i), there exists a sequence $(S_{m})_{m}$ in $G\cap (\mathcal{S}_{n}\backslash \mathcal{T}_{n})$ such that
  $\underset{m\longrightarrow +\infty}{lim}S_{m}=S=(a, -1)$. So $S(x)=-x+a\in \overline{G(x)}$.\
\\
It follows that $\alpha x +a\in \overline{G(x)}$ and so $$\left(\Lambda_{G}\cap\{-1,1\}\right)x+E_{G}\subset \overline{G(x)}.$$

This proves that $\overline{\Lambda_{G}}.x+E_{G}\subset\overline{G(x)}$.\
\\
\\
 Conversely,\ let's prove that $G(x)\subset \Lambda_{G}.x+E_{G}$.\ Let $f\in G$.
 \\
 $\bullet$  Suppose that   $f=(a,\lambda)\in G\backslash \mathcal{S}_{n}$.  By Lemma ~\ref{L:2}.(i), $f(0)=(1-\lambda)a\in E_{G}$ since $E_{G}$ is a vector space.
\\
 Then $f(x)=\lambda (x-a) +a=\lambda x+ (1-\lambda)a\in \Lambda_{G}.x+E_{G}$.\
 \\
 $\bullet$  Suppose that   $f=(a,\varepsilon)\in G\cap \mathcal{S}_{n}$, so $f(x)=\varepsilon x+a\in\Lambda_{G}.x+E_{G}$, since
 by Lemma ~\ref{L:2}.(i), $f(0)=a\in E_{G}$.\
 \\
 It follows that  $G(x)\subset \Lambda_{G}.x+E_{G}$. Therefore
  $\overline{G(x)}\subset \overline{\Lambda_{G}}.x+E_{G}$. Hence $\overline{G(x)}=\overline{\Lambda_{G}}.x+E_{G}$.
\end{proof}
\bigskip

\begin{lem}\label{L:5} If there exists  $\lambda,\mu\in\Lambda_{G}$  such that $\lambda\mu<0$ and
$\frac{log|\lambda|}{log|\mu|}\notin\mathbb{Q}$, then
$\overline{\Lambda_{G}}=\mathbb{R}$.
\end{lem}
\medskip

\begin{proof} Suppose that $\lambda<0<\mu$. Let $H_{+}:=\{\lambda^{2p}\mu^{2q},  p,q\in
\mathbb{Z}\}$ and $H_{-}:=\lambda.H_{+}$. See that $H_{-}\subset \Lambda_{G}$ and so  $H_{+}\cup H_{-}\subset
\Lambda_{G}$. Set  $f: ]0,+\infty[\longrightarrow \mathbb{R}$, the
homeomorphism defined by  $f(x)=log x$, so
$f(H_{+}):=\mathbb{Z}+\frac{log|\lambda|}{log|\mu|}\mathbb{Z}$. As
$\frac{log|\lambda|}{log|\mu|}\notin\mathbb{Q}$ then $f(H_{+})$ is
dense in $\mathbb{R}$, so $H_{+}$ and $H_{-}$ are dense
respectively in $]0,+\infty[$ and  in $]-\infty,0[$. We deduce
that $\overline{\Lambda_{G}}=\mathbb{R}$.
\end{proof}
\bigskip

\section{\textbf{Some results for non abelian subgroup of $\mathcal{S}_{n}$}}
In this case, $G$ is a non abelian subgroup of $\mathcal{S}_{n}$,
then it contains necessarily  an affine symmetry. In the
following, recall that $G_{1}=G\cap \mathcal{T}_{n}$ and every $f\in
\mathcal{S}_{n}$ is denoted by $f=(a,\varepsilon)$, where $f: x\longmapsto\varepsilon x +a$.
Denote by   $\delta_{G}:=\{f(0),\ \ f\in G\cap (\mathcal{S}_{n}\backslash
\mathcal{T}_{n}) \}$.
\
\\
We use the
following lemmas and propositions to prove Theorem ~\ref{T:1} and
above Corollaries:

\begin{lem}\label{L:14} Let $G$ be a non abelian  subgroup of $\mathcal{S}_{n}$. Then:
\begin{itemize}
\item[(i)] $G_{1}(0)$
is an additif subgroup of $\mathbb{R}^{n}$.
\item[(ii)] \ $\delta_{G}\neq\emptyset$.
 \item [(iii)]  For every  $f\in G\backslash \mathcal{T}_{n}$,\ we have \  $f(G_{1}(0))= \delta_{G}$ and $f(\delta_{G})= G_{1}(0)$.
\end{itemize}
\end{lem}
\bigskip

\begin{proof}The proof of (i) is obvious.
\\
(ii) If  $\delta_{G}=\emptyset$  then $G\cap (\mathcal{S}_{n}\backslash
\mathcal{T}_{n})=\emptyset$, so  $G$ is a
subgroup of $\mathcal{T}_{n}(\mathbb{R})$,  hence  $G$ is abelian,
a contradiction.\\
\\
(iii) Let  $f\in G\backslash \mathcal{T}_{n}$,
 $b\in G_{1}(0)$ and  $g=(b,1)\in G_{1}$ . Then for every  $x\in \mathbb{R}^{n}$  we have
$f\circ g(x)=f(x+b)=-x-b+a$, so  $f\circ g=(-b+a,-1)$. Hence
$f(b)=f\circ g(0)=-b+a\in \delta_{G}$.\\ Conversely, let
$b\in\delta_{G}$ and $g=(b,-1)\in G\backslash
\mathcal{T}_{n}$ such that $g(0)=b$. We have $f\circ
g(x)=f(-x+b)=x-b+a$, \ so $f\circ g=(-b+a,1)\in G_{1}$, thus
$c=-b+a\in G_{1}(0)$. Hence $b=-c+a=f(c)$ and so $b\in
f(G_{1}(0))$. It follows that $f(G_{1}(0))=\delta_{G}$. As
$f^{-1}=f$ so $f(\delta_{G})= G_{1}(0)$.
\end{proof}

\medskip

\begin{prop} \label{p:3}Let  $G$ be a non abelian subgroup of  $\mathcal{S}_{n}$,  $a\in\delta_{G}$ and  $x\in\mathbb{R}^{n}$. Then:
 $$\overline{G(x)}=(x+\overline{G_{1}(0)})\cup(-x+a+ \overline{G_{1}(0)}).$$
\end{prop}
\bigskip

\begin{proof} Let  $G$ be a non abelian subgroup of  $\mathcal{S}_{n}$ and $x\in \mathbb{R}^{n}$. We have $$G(x)=\{g(x)=\varepsilon x+b, \
g=(b,\varepsilon)\in
G\}=\left(x+G_{1}(0)\right)\cup\left(-x+\delta_{G}\right). \ \
\mathrm{So}$$
$$\overline{G(x)}=\left(x+\overline{G_{1}(0)}\right)\cup\left(-x+\overline{\delta_{G}}\right).\
\ \ \ \ (1) $$\
\\
Since $G$ is non abelian then $G\backslash\mathcal{T}_{n}\neq
\emptyset$, so let $f=(a,-1)\in
G\backslash\mathcal{T}_{n}$ with  $a\in\delta_{G}$. By Lemma ~\ref{L:14}.(iii) we
have $\delta_{G}=f(G_{1}(0))$.  By Lemma ~\ref{L:14}.(i), $G_{1}(0)$ is an additif
subgroup of $\mathbb{R}^{n}$  then $f(G_{1}(0))=G_{1}(0)+a$, so
$\delta_{G}=G_{1}(0)+a$.  Hence
$-x+\overline{\delta_{G}}=-x+\overline{G_{1}(0)}+a$. By $(1)$  we
conclude that
$$\overline{G(x)}=\left(x+\overline{G_{1}(0)}\right)\cup\left(-x+a+\overline{G_{1}(0)}\right).$$
\end{proof}

\section{{\bf Proof of main results}}
\
\\
{\it Proof of Theorem ~\ref{T:1}.} Let $a\in E_{G}$ and $G'=T_{-a}\circ G\circ T_{a}$. By Lemma ~\ref{L:1}.(ii),
 $E_{G'}=T_{-a}(E_{G})$, so $E_{G'}$ is a vector subspace of $\mathbb{R}^{n}$. Then :
\\
$\bullet$ {\it Proof of $(1).(i)$:} One has $\Gamma_{G}\neq\emptyset$ since $G\backslash \mathcal{S}_{n}\neq\emptyset$.
Then by Lemma ~\ref{L:00}, $0\in\overline{\Lambda_{G}}$ and by Lemma~\ref{L:2}.(ii), $E_{G}$ is $G$-invariant.
As $G\backslash \mathcal{S}_{n}\neq\emptyset$, there exist $b,c\in\Gamma_{G}\subset E_{G}$, with $b\neq c$ (Lemma~\ref{L:6}.(iii)),  so $\mathrm{dim}(E_{G})\geq1$. \
\\
$\bullet$ {\it Proof of $(1).(ii)$:} By Proposition ~\ref{p:1},
 $\overline{G'(x-a)}=E_{G'}$, for every $x\in E_{G}$. So $T_{-a}(\overline{G(x)})=E_{G'}$,
  it follows that $\overline{G(x)}=T_{a}(E_{G'})=E_{G}$. So the proof of (1)(i) is complete.
  \
  \\
  $\bullet$ {\it Proof of $(1).(iii)$:} By Proposition ~\ref{p:2},
 $\overline{G'(x-a)}=\overline{\Lambda_{G'}}.(x-a)+E_{G}'$, for every $x\in U$. So by Lemma~\ref{L:1},(ii),
 $T_{-a}(\overline{G(x)})=\overline{\Lambda_{G}}.(x-a)+E_{G}-a$,
  it follows that $\overline{G(x)}=\overline{\Lambda_{G}}.(x-a)+E_{G}$. So the proof of (1)(ii) is complete.
\
\\
$\bullet$ {\it Proof of $(2)$:} The proof of (2) results from  Lemma ~\ref{L:14}.(i) and Proposition ~\ref{p:3}, since $H_{G}=G_{1}(0)$.
\hfill{$\Box$}
\
\\
\\
We will use the following Lemmas to prove Corollary ~\ref{C:1}.
\begin{lem}\label{L:15} Let $G$ be a non abelian
subgroup of $\mathcal{H}(n, \mathbb{R})$ with $G\backslash
\mathcal{S}_{n}\neq\emptyset$, then for every $x\in
U$ we have $\overline{G(x)}=\overline{G(y)}$.
\end{lem}
\bigskip

\begin{proof} Suppose that $E_{G}$ is a vector space (leaving to replace $G$ by $G'=T_{-a}\circ G \circ T_{a}$ for some $a\in E_{G}$,
 and by Lemma ~\ref{L:1}.(ii), $E_{G'}=T_{-a}(E_{G})$ is a vector space). Let $x\in U$ and $y\in \overline{G(x)}\cap U$.
  By Theorem ~\ref{T:1}.(1).(ii), there exists $a\in E_{G}$ such that
 $\overline{G(x)}=\overline{\Lambda_{G}}(x-a)+E_{G}$. Since $E_{G}$ is a vector space and $a\in E_{G}$
 then $\overline{G(x)}=\overline{\Lambda_{G}}x+E_{G}$. In the same way, $$\overline{G(y)}=\overline{\Lambda_{G}}y+E_{G}, \ \ \ \ \ \ \ \ \ (1).$$ See that $\overline{G(x)}\cap U= (\overline{\Lambda_{G}}\backslash\{0\})x+E_{G}$.
  Write $y=\alpha x+b$, where $\alpha\in\overline{\Lambda_{G}}\backslash\{0\}$ and $b\in E_{G}$. So by (1),
$$\overline{G(y)}=\overline{\Lambda_{G}}y+E_{G}=\overline{\Lambda_{G}}(\alpha x+b)+E_{G}=\alpha\overline{\Lambda_{G}} x+E_{G}.$$ Since $\alpha\in\Lambda_{G}$ and by Lemma ~\ref{L:00},
$\overline{\Lambda_{G}}\backslash\{0\}$ is a subgroup of $\mathbb{R}^{*}$, then $\alpha \overline{\Lambda_{G}}=\overline{\Lambda_{G}}$.
 Therefore $\overline{G(y)}=\overline{\Lambda_{G}}x+E_{G}=\overline{G(x)}.$
\end{proof}
\medskip

\begin{lem}\label{L:16}  Let $G$ be a non abelian
subgroup of $\mathcal{H}(n, \mathbb{R})$ such that  $E_{G}$
 is a vector subspace of $\mathbb{R}^{n}$. Let  $x\in U$  then
 the vector subspace   $H_{x}=\mathbb{R}x\oplus E_{G}$  of
 $\mathbb{R}^{n}$  is $G$-invariant.
\end{lem}
\medskip

\begin{proof} Let  $x\in
\mathbb{R}^{n}\backslash E_{G}$  and
$H_{x}=\mathbb{R}x+E_{G}$. Let  $f\in G$ having the form
 $f(z)=\lambda z+a$,  $z\in \mathbb{R}^{n}$, then by Lemma ~\ref{C:1}.(i),  $a=f(0)\in
E_{G}$. For every $\alpha\in \mathbb{R}$, $b\in E_{G}$,
we have  $f(\alpha x +b)=\lambda(\alpha x +b)+a=\lambda\alpha
x+\lambda b+a$.  Since $E_{G}$ is a vector space, then
$\lambda b+a\in E_{G}$ and so $f(\alpha x +b)\in H_{x}$.
\end{proof}
\bigskip

\begin{proof}[Proof of Corollary ~\ref{C:1}]\ \\
$\bullet$ {\it The proof of $(i)$:} The proof results from Lemma ~\ref{L:15}.\
\\
$\bullet$ {\it The proof of $(ii)$:} As $G\backslash\mathcal{S}_{n}\neq \emptyset$, then by Lemma~\ref{L:00}, $0\in \overline{\Lambda_{G}}$. So
 the proof of (ii) results from Theorem 1,1.(1).(ii).
\\
$\bullet$ {\it The proof of $(iii)$:}   Suppose that
$E_{G}$ is a vector subspace of $\mathbb{R}^{n}$ (leaving, by Lemma ~\ref{C:1}.(ii), to
replace $G$ by $G'=T_{-a}\circ G \circ T_{a}$, for some $a\in E_{G}$).

Recall that $U= \mathbb{R}^{n}\backslash E_{G}$ and let  $x, y\in U$
 with $x\neq y$.  Denote by  $H_{x}=\mathbb{R}.x\oplus E_{G}$
 and by   $H_{y}=\mathbb{R}.y\oplus E_{G}$. By lemma ~\ref{L:16} we have
 $H_{x}$  and  $H_{y}$  are $G$-invariant. Let   $\varphi:\ H_{x}\longrightarrow
 H_{y}$  be the homeomorphism defined by $\varphi(\alpha x+v)=\alpha y+v$
 for every $\alpha\in \mathbb{R}$  and  $v\in E_{G}$. For
 every  $f\in G$, with the form $f(z)=\lambda
 z+a$, $z\in\mathbb{R}^{n}$,  then by Lemma ~\ref{L:2}.(i), $a=f(0)\in E_{G}$  and so  $\varphi(f(x))=\varphi(\lambda x+a)=\lambda
 y+a=f(y)$. It follows that  $\varphi(G(x))=G(y)$.
\end{proof}
\
\\
\\
{\it Proof of Corollary ~\ref{C:2}.}\
\\
 $\bullet$ {\it The proof of $(i)$:} From
Corollary ~\ref{C:1}.(ii),  the closure of every orbit of $G$ contains $E_{G}$ and by Theorem
~\ref{T:1}.(1), we have  $\mathrm{dim}(E_{G})\geq 1$, so $G$ has no periodic orbit. Moreover, if $G$ is countable
then every orbit $O$ is also countable, hence $O$ can not be closed.\
\\
$\bullet$ {\it The proof of $(ii)$:}  Let $x\in\mathbb{R}^{n}$ and $y\in
\overline{G(x)}$. By Proposition ~\ref{p:3} we have
$\overline{G(x)}=\left(x+\overline{G_{1}(0)}\right)\cup\left(-x+a+\overline{G_{1}(0)}\right)$.
Suppose that $y\in (x+\overline{G_{1}(0)})$   then $y=x+b$ for
some $b\in\overline{G_{1}(0)}$. By Lemma ~\ref{L:14}.(i), $G_{1}(0)$ is an additif group, so $b+G_{1}(0)=G_{1}(0)$.
Therefore, by Proposition ~\ref{p:3} we have

\begin{align*}
\overline{G(y)} & =\left(x+b+\overline{G_{1}(0)}\right)\cup\left(-x-b+a+\overline{G_{1}(0)}\right)\\
\ & =\left(x+\overline{G_{1}(0)}\right)\cup\left(-x+a+\overline{G_{1}(0)}\right)=\overline{G(x)}.
\end{align*}
\
\\
 The same proof is used if $y\in
(x+a+\overline{G_{1}(0)})$.  \hfill{$\Box$}
\
\\
\\
\\
{\it Proof of Corollary ~\ref{C:3}.} Let $G$ is a non abelian subgroup of  $\mathcal{H}(n, \mathbb{R})$  such that  $G\backslash
\mathcal{S}_{n}\neq\emptyset$. Suppose that  $E_{G}$ \ is a
vector subspace of $\mathbb{R}^{n}$ (leaving, by Lemma ~\ref{L:1}.(ii), to replace $G$ by
$T_{-a}\circ G \circ T_{a}$, for some $a\in E_{G}$.)
\\
\\
$\bullet$ Let's prove that $(1)$ and $(2)$ are equivalent: if $\overline{G(x)}=\mathbb{R}^{n}$, for some $x\in \mathbb{R}^{n}$, so $x\in U$.
Let $y\in U$, then by Corollary ~\ref{C:1}.(i), $\overline{G(y)}\cap U=\overline{G(x)}\cap U=U$. Since $U$ is dense in
$\mathbb{R}^{n}$, $\overline{G(y)}=\mathbb{R}^{n}$. Conversely, the proof is obvious.
\\
\\
 $\bullet$ $(3).(i)\Longrightarrow (1)$: If
$E_{G}=\mathbb{R}^{n}$  then by Theorem ~\ref{T:1}.(1).(i) we have
$\overline{G(x)}=E_{G}$, for every $x\in E_{G}$. So $G$
has a dense orbit.
\\
\\
$\bullet$ $(3).(ii)\Longrightarrow (1)$: If \ $\Lambda_{G}$  is
dense in  $\mathbb{R}$  then by Theorem ~\ref{T:1}.(1).(ii) we have
$\overline{G(x)}=\overline{\Lambda_{G}}x+E_{G}$,   for every
$x\in U$. So $G$ has a dense orbit.
\\
\\
$\bullet$ $(1)\Longrightarrow (3)$: Suppose that $G$ has a
dense orbit $G(x)$, for some  $x\in \mathbb{R}^{n}$.  There are
tow cases:
\\
- If  $E_{G}=\mathbb{R}^{n}$,  then we obtain $(3).(i)$.
\\
- If  $E_{G}\neq\mathbb{R}^{n}$  then $\mathrm{dim}(E_{G})\leq
n-1 $, so  and $U\neq\emptyset$, hence  $x\in U$. By Theorem ~\ref{T:1}.(1).(ii) we
have  $\overline{G(x)}=\mathbb{R}x+E_{G}$,  so
$dim(E_{G})=n-1$  and  $\Lambda_{G}$  is dense in
$\mathbb{R}$. Then $(3).(ii)$ follows.\hfill{$\Box$}
\
\\
\\
We use the following Lemma to prove Corollary ~\ref{C:4}:

\begin{lem}\label{L:17} Let $H$  be an additif subgroup of $\mathbb{R}^{n}$. Then
$$\overset{\circ}{\overline{H}}\neq\emptyset \ \ \ if\ and\ only \ if \ \ \ \ \overline{H} = \mathbb{R}^{n}$$
\end{lem}

\begin{proof}
Suppose that $\overset{\circ}{\overline{H}}\neq\emptyset$ and let $a\in \overset{\circ}{\overline{H}}\neq\emptyset$. Then there exists $\varepsilon >0$ such that
 $B_{(a,\varepsilon)}\subset \overline{H}\neq\emptyset$,
 where $B_{(a,\varepsilon)}=\{x\in\mathbb{R}^{n}:\ \|x-a\|<\varepsilon\}$ and $\|.\|$ is the euclidian norm. Since
 \ $\overline{H}\neq\emptyset$  is an additif group, it follows that
 $B_{(0,\varepsilon)} = T_{-a}\left(B_{(a,\varepsilon)}\right)\subset \overline{H}$. Moreover, we also have $B_{(0,m\varepsilon)} =
mB_{(0,\varepsilon)}\subset\overline{H}$, for every $m\in\mathbb{N}^{*}$. As
   $\mathbb{R}^{n}=\underset{m\in\mathbb{N}^{*}}{\bigcup}B_{(0,m\varepsilon)}\subset\overline{H}$, it follows that
   $\overline{H}=\mathbb{R}^{n}$. Conversely, the proof is obvious.
\end{proof}
\
\\
\\
{\it  Proof of Corollary ~\ref{C:4}.} Let $G$ be a non abelian
subgroup of  $\mathcal{S}_{n}$. By Lemma ~\ref{L:2}.(i), one has
$\delta_{G}\neq \emptyset$. Let $a\in \delta_{G}$ and $f=(a,-1)\in
G$.\\
$\bullet$ First, by Corollary ~\ref{C:2}.(ii) we prove that\ $(i)$ , $(ii)$ and $(iii)$ are
equivalent.\\
$\bullet$ Second, let's prove that $(iii)$ and $(iv)$ are equivalent:
Suppose that  $\overline{G(0)}=\mathbb{R}^{n}$.  By Proposition
~\ref{p:3} we have $\overline{G(0)}= \overline{G_{1}(0)}\cup
(a+\overline{G_{1}(0)})$. Since
$\overset{\circ}{\overline{G(0)}}\neq\emptyset$ then
$\overset{\circ}{\overline{G_{1}(0)}}\neq\emptyset$. By Lemma
~\ref{L:14}.(i), $H_{G}=G_{1}(0)$ is an additive  subgroup of $\mathbb{R}^{n}$
then by Lemma ~\ref{L:17}, $\overline{H_{G}}=\mathbb{R}^{n}$. Conversely, if
$\overline{H_{G}}=\mathbb{R}^{n}$  then  by By Proposition
~\ref{p:3} we have $\overline{G(0)}= \overline{G_{1}(0)}\cup
(a+\overline{G_{1}(0)})=\mathbb{R}^{n}$. \hfill{$\Box$}
\
\\
\\
{\it Proof of \ Corollary ~\ref{C:5}.} For $n=1$,  $G$ is a non abelian group of affine maps of $\mathbb{R}$.\\
$\bullet$ {\it The proof of $(i)$:} If  $G\backslash\mathcal{S}_{1}\neq\emptyset$,
then by Theorem ~\ref{T:1}.(1),  $E_{G}$ is a $G$-invariant affine
subspace of $\mathbb{R}$ with dimension $p= 1$ such that every
orbit of $E_{G}$ is dense in it. In this case $E_{G}=\mathbb{R}$.
\\
$\bullet$ {\it The proof of $(ii)$:}  If  $G\subset\mathcal{S}_{1}$, then by
Theorem ~\ref{T:1}.(2), $H_{G}$  is a $G$-invariant closed subgroup  of
$\mathbb{R}$ and there exists $a\in E_{G}$ such that for every $x\in
\mathbb{R}$, we have $\overline{G(x)}=(x+H_{G})\cup(-x+a+H_{G})$. Then
there are two cases: \\
$\diamond$ If $H_{G}$ is dense in $\mathbb{R}$, so every
orbit of $G$ is dense in $\mathbb{R}$.\\
$\diamond$  If $H_{G}$ is  discrete then every orbit is closed and discrete.
\bigskip

\section{{\bf Examples }}
\begin{exe} Let  $G$  be a subgroup of  $\mathcal{H}(2, \mathbb{R})$  generated by  $f_{1}=(a_{1},\alpha_{1})$
and \ $f_{2}=(a_{2},\alpha_{2})$  and  $f_{3}=(a_{3},
\alpha_{3})$, where $\alpha_{k}\neq 1$,  for every $1\leq k \leq
3$  and  $a_{1}=\left[\begin{array}{c}
               \sqrt{2} \\
               0
             \end{array}
\right]$,  $a_{2}=\left[\begin{array}{c}
               0 \\
               1
             \end{array}
\right]$  and  $a_{3}=\left[\begin{array}{c}
               -\sqrt{3} \\
               -\sqrt{2}
             \end{array}
\right]$. Then every orbit of $G$ is dense in  $\mathbb{R}$.
\
\\
\\
Indeed, by Lemma ~\ref{L:6}.(i), $G$ is non abelian. By Proposition ~\ref{p:1}, for
every $x\in E_{G}$, we have $\overline{G(x)}=E_{G}$. In
this case, by Remark~\ref{r:3}, $E_{G}=\mathbb{R}^{2}$ so every orbit of
$G$ is dense in  $\mathbb{R}^{2}$.
\end{exe}
\bigskip

\begin{exe}Let  $(a_{1},\dots,a_{n})$  be a basis of  $\mathbb{R}^{n}$ \ and $a=\underset{k=1}{\overset{n}{\sum}}\alpha_{k}a_{k}$,
 with $\underset{k=1}{\overset{n}{\sum}}\alpha_{k}a_{k}\neq 1$, then for every $t>1$,    the subgroup  $G$  of
$\mathcal{H}(n, \mathbb{R})$  generated by
 $\{f=(a,t), \ T_{a_{k}},  2\leq k \leq n\}$  is minimal. (i.e. every orbit of  $G$  is dense
 in  $\mathbb{R}^{n}$).
\
\\
\\Indeed;  By Remark ~\ref{r:3}, we have  $E_{G}=\mathbb{R}^{n}$
and by Proposition ~\ref{p:1}, every orbit of $G$ is dense in
$\mathbb{R}^{n}$.
\end{exe}
\medskip

\begin{exe} Let   $(a_{1},\dots,a_{n})$ \ be a basis of $\mathbb{R}^{n}$ and  $\lambda\in \mathbb{R}\backslash\{0,1\}$.
  Then every orbit of the group generated by $T_{a_{1}},\dots,T_{a_{n}}, \ \lambda Id$  is dense in $\mathbb{R}^{n}$.
\
\\
\\
  Indeed, By Remark ~\ref{r:3} we have $E_{G}=\mathbb{R}^{n}$ and by Proposition ~\ref{p:1} every orbit of $G$ is dense in $\mathbb{R}^{n}$.
 \end{exe}
  \medskip

  \begin{exe}\label{ex:4} Let   $a\in \mathbb{R}^{n}$ and  $G$ be the group generated by $f=T_{a}$,
  $g=(a,-1)$ and $h=T_{\sqrt{2}a}$. Then
  for every $x\in \mathbb{R}^{n}\backslash \mathbb{R}a$, we have
  $G(0)$ and $G(x)$ are not homeomorphic.
  \end{exe}
\medskip

  \begin{proof} Remark that for every  $\varphi\in G_{1}$,
   there exist $n_{1},m_{1},p_{1},\dots,n_{r},m_{r},p_{r}\in \mathbb{Z}$ such that $\varphi=(f^{n_{1}}\circ g^{m_{1}}\circ h^{p_{1}})\circ
  \dots \circ (f^{n_{r}}\circ g^{m_{r}}\circ h^{p_{r}})$, for some
  $r\in \mathbb{N}^{*}$.\\
\\
  $\bullet$ First, let's show by induction on $r\geq 1$ that
  $$\varphi(0)\in(\mathbb{Z}+\sqrt{2}\mathbb{Z})a\  \ \ (i).$$
\\
For $r=1$, we have \begin{align*}
\varphi(0)& =f^{n_{1}}\circ g^{m_{1}}\circ
h^{p_{1}}(0)\\
\ & =-p_{1}\sqrt{2}a+m_{1}a+n_{1}a\\
\ & =(m_{1}+n_{1}+\sqrt{2}p_{1})a.\
\end{align*}
\\
So $\varphi(0)\in\in
(\mathbb{Z}+\sqrt{2}\mathbb{Z})a$.\
\\
  Suppose property (i) is true up to order $r-1$. If $$\varphi=\left(f^{n_{1}}\circ g^{m_{1}}\circ
  h^{p_{1}}\right)\circ\left(f^{n_{2}}\circ g^{m_{2}}\circ h^{p_{2}}\circ
  \dots \circ f^{n_{r}}\circ g^{m_{r}}\circ h^{p_{r}}\right),$$ then by
  induction property there exists $p,q\in \mathbb{Z}$ such that $$f^{n_{2}}\circ g^{m_{2}}\circ h^{p_{2}}\circ
  \dots \circ f^{n_{r}}\circ g^{m_{r}}\circ
  h^{p_{r}}(0)=(p+\sqrt{2}q)a.$$ So $\varphi(0)= f^{n_{1}}\circ g^{m_{1}}\circ
  h^{p_{1}}((p+\sqrt{2}q)a),$ thus  $$\varphi(0)=
  \begin{cases}
  -\left((p+\sqrt{2}q)a+p_{1}a\right)+a+n_{1}a, &
  \mathrm{if} \ \
  m_{1}\ \mathrm{is \ odd}, \\
    \left((p+\sqrt{2}q)a+p_{1}a\right)+n_{1}a, &
  \mathrm{if} \ \ m_{1}\ \mathrm{is \ even}.
  \end{cases}$$
   Hence,
  $\varphi(0)\in(\mathbb{Z}+\sqrt{2}\mathbb{Z})a$.\\

  It follows that $$G_{1}(0)\subset(\mathbb{Z}+\sqrt{2}\mathbb{Z})a\ \ \ \ \ \ (1)$$\
  \\
\\
  $\bullet$ Second, we will proof that
  $G_{1}(0)=(\mathbb{Z}+\sqrt{2}\mathbb{Z})a$. let $p,q\in \mathbb{Z}$, we have $f^{p}\circ
  h^{q}=T_{(p+\sqrt{2}q)a}$, thus  $f^{p}\circ
  h^{q}(0)=(p+\sqrt{2}q)a\in G_{1}(0)$. It follows by (1) that $G_{1}(0)=(\mathbb{Z}+\sqrt{2}\mathbb{Z})a$.   With the same proof we can
  show that $\delta_{G}=G_{1}(0)$.\\
\\
  $\bullet$ Thirdly, by Proposition ~\ref{p:3} for every $x\in\mathbb{R}^{n}$, we
  have $\overline{G(x)}=(x+\overline{G_{1}(0)})\cup (-x+\overline{G_{1}(0)})$. Therefore
  $\overline{G(0)}=\overline{G_{1}(0)}=\overline{(\mathbb{Z}+\sqrt{2}\mathbb{Z})}a=\mathbb{R}a$
  and it is connected. But $\overline{G(x)}=(x+\mathbb{R}a)\cup
  (-x+\mathbb{R}a)$, is not connected for every $x\in
  \mathbb{R}^{n}\backslash\mathbb{R}a$. Hence $G(0)$ and $G(x)$
  can not be homeomorphic.
 \end{proof}
 \bigskip

 \begin{rem} Remark that the form of $\varphi$ used in the proof of Example ~\ref{ex:4}
 is general of every $\varphi \in G$ and the order
 $(f^{n_{k}}\circ g^{m_{k}}\circ h^{p_{k}})$ is not particular of $\varphi$. For example, if
 $\varphi=g^{m}\circ f^{n}\circ h^{p}$, we write
 $$\varphi=(f^{n_{1}}\circ g^{m_{1}}\circ h^{p_{1}})\circ(f^{n_{2}}\circ g^{m_{2}}\circ h^{p_{2}})\circ(f^{n_{3}}\circ g^{m_{3}}\circ h^{p_{3}})$$
  with $n_{1}=p_{1}=m_{2}=p_{2}=n_{3}=m_{3}=0$,  $m_{1}=m$, $n_{2}=n$ and $p_{3}=p$.
 \end{rem}

\bibliographystyle{amsplain}

\begin{thebibliography}{9}



\bibitem{VMS} Vitaly.B, Michal.M and Samuel.S,
\emph{Affine actions of a free semigroup on the real line}, Ergod.
Th. and Dynam. Sys. (2006), 26, 1285–1305.

\bibitem{MJVH} Mohamed.J,\emph{A generalization of Dirichlet approximation theorem
for the affine actions on real line}, Journal of Number Theory 128
(2008) 1146–1156.
\end{thebibliography}
\vskip 0,4 cm

\end{document}